\documentclass[11pt]{article}
\usepackage{amsmath,fullpage,amssymb,amsthm,hyperref,enumerate}
\usepackage{tikz,color}
\usetikzlibrary{decorations.pathmorphing}

\newtheorem{theorem}{Theorem}[section]
\newtheorem{proposition}[theorem]{Proposition}
\newtheorem{lemma}[theorem]{Lemma}
\newtheorem{corollary}[theorem]{Corollary}
\newtheorem{definition}[theorem]{Definition}   
\newtheorem{conjecture}[theorem]{Conjecture}  
\newtheorem{problem}[theorem]{Problem}  

\theoremstyle{remark}

\newcommand\G{\mathcal{GD}}
\newcommand\DP{\mathcal D}
\newcommand\B{\mathcal B}
\renewcommand\P{\mathcal P}
\newcommand\UB{\mathcal U}
\newcommand\M{\mathcal M}
\newcommand\GM{\mathcal{GM}}
\newcommand\WQ{\mathcal{W}}
\newcommand\bij{\phi}
\newcommand\bijw{\omega}
\DeclareMathOperator\ds{ds}

\DeclareMathOperator\dsh{ds\ulcorner}
\renewcommand\sp{\operatorname{sp}}
\DeclareMathOperator\hm{hm}
\newcommand\Cat{C}

\newcommand\sv{\operatorname{sv}}
\newcommand\ret{\operatorname{ret}}
\newcommand\diag{\operatorname{diag}}
\newcommand{\card}[1]{{\lvert #1 \rvert}} 	
\newcommand\pho{\operatorname{ph}_1}
\newcommand\NE{\textnormal{NE}}
\newcommand\NW{\textnormal{NW}}
\newcommand\SE{\textnormal{SE}}
\newcommand\SW{\textnormal{SW}}
\newcommand\vD{\widehat{D}}
\newcommand\vR{\widehat{R}}
\newcommand{\E}{\mathbb{E}}
\newcommand{\V}{\mathbb{V}}
\renewcommand{\Pr}{\mathbb{P}}
\newcommand{\Ral}{\mathcal R}
\newcommand{\Hn}{\mathcal H}

\newcommand\U{-- ++(1,1) circle(1.2pt)}
\newcommand\D{-- ++(1,-1) circle(1.2pt)}
\renewcommand\H{-- ++(1,0) circle(1.2pt)}
\newcommand\drawHtwo[1][black]{\draw[very thick,decorate,decoration={snake,amplitude=.3mm,segment length=1mm},#1] (last)-- ++(1,0) coordinate(last);}
\newcommand\Hone{++(0,.07) -- ++(1,0)++(-1,-.14) -- ++(1,0)++(0,.07) circle(1.2pt)}

\newcommand\N{-- ++(0,1) circle(1.2pt)}
\renewcommand\S{-- ++(0,-1) circle(1.2pt)}

\begin{document}

\title{The degree of symmetry of lattice paths}

\author{Sergi Elizalde\thanks{Department of Mathematics, Dartmouth College, Hanover, NH 03755, USA. {\tt sergi.elizalde@dartmouth.edu}.}}

\date{}

\maketitle

\begin{abstract}
The degree of symmetry of a combinatorial object, such as a lattice path, is a measure of how symmetric the object is. It typically ranges from zero, if the object is completely asymmetric, to its size, if it is completely symmetric.
We study the behavior of this statistic on Dyck paths and grand Dyck paths, with symmetry described by reflection along a vertical line through their midpoint; partitions, with symmetry given by conjugation; and certain compositions interpreted as bargraphs.
We find expressions for the generating functions for these objects with respect to their degree of symmetry, and their semilength or semiperimeter, deducing in most cases that, asymptotically, the degree of symmetry has a Rayleigh or half-normal limiting distribution. The resulting generating functions are often algebraic, with the notable exception of Dyck paths, for which we conjecture that it is D-finite (but not algebraic), based on a functional equation that we obtain using bijections to walks in the plane.
\end{abstract}

\section{Introduction}

Many combinatorial classes are naturally endowed with a reflection operation, namely, an involution.
For those classes, it is natural to study the subset of objects that are {\em symmetric}, that is, invariant under such reflection. Some examples of symmetric combinatorial objects in the literature are symmetric Dyck paths~\cite{DDS}, symmetric grand Dyck paths, self-conjugate partitions~\cite[Prop.\ 1.8.4]{EC1}, symmetric plane partitions~\cite{Mac,Stanley}, symmetric planar maps~\cite{Albenque}, centrally symmetric dissections of a polygon~\cite{Simion}, palindromic compositions~\cite{HB}, and symmetric binary trees, to mention a few.

In this paper we propose a refinement of the notion of symmetric objects, by introducing a type of combinatorial statistic that we call the {\em degree of symmetry}, which measures how close the object is to being symmetric. We study the degree of symmetry of lattice paths, as well as related combinatorial objects such as partitions and bargraphs. In other concurrent work~\cite{DEasym}, a related statistic (specifically, the complementary notion of {\em degree of asymmetry}) is studied as it applies to sequences of integers, matchings, and trees.

Our focus is on the following kind of lattice paths, often called {\em grand Dyck paths}. 
These are paths in the plane with up-steps $U=(1,1)$ and down-steps $D=(1,-1)$ from $(0,0)$ to $(2n,0)$. We call $n$ the {\em semilength} of the path, and we denote by $\G_n$ the set of grand Dyck paths of semilength~$n$. Let $\DP_n$ be the subset of those that do not go below the $x$-axis, which are called {\em Dyck paths}.

Given a path $P\in\G_n$, we view its steps as segments in the plane, which we denote by $\bar{p}_1,\bar{p}_2,\dots,\bar{p}_{2n}$ from left to right. For example, $\bar{p}_1$ has endpoints $(0,0)$ and $(1,\pm 1)$, and $\bar{p}_{2n}$ has endpoints $(2n-1,\pm1)$ and $(2n,0)$. For $i\in[n]=\{1,2,\dots,n\}$, we say that $P$ is symmetric in position $i$ 
(or that $\bar{p}_i$ is a {\em symmetric step}) if $\bar{p}_i$ and $\bar{p}_{2n+1-i}$ are mirror images of each other with respect to the reflection along the vertical line $x=n$. 
We define the {\em degree of symmetry} of $P$, denoted by $\ds(P)$, as the number of $i\in[n]$ such that $P$ is symmetric in position $i$.
See Figure~\ref{fig:symD} for an example.

\begin{figure}[htb]
\centering
   \begin{tikzpicture}[scale=0.5] 
     \draw[dotted](-1,0)--(23,0);
     \draw[dotted](0,-3)--(0,2);
     \draw[dashed,thick](11,-3)--(11,2);
     \draw[thick,blue](0,0) circle(1.2pt) \U\D\D\U\D\U\U\D\U\D\U  \D\D\D\U\U\D\U\U\D\D\U;
     \draw[red,ultra thick](4,0)--(5,-1)--(6,0) (10,0)--(11,1)--(12,0) (16,0)--(17,-1)--(18,0);
    \end{tikzpicture}\medskip
       
    \begin{tikzpicture}[scale=0.5]
     \draw[dotted](-1,0)--(15,0);
     \draw[dotted](0,-1)--(0,4);
     \draw[dashed,thick](7,-1)--(7,4);
     \draw[thick,blue](0,0) \U\U\U\D\D\U\D\D\U\U\D\U\D\D;
     \draw[red,ultra thick](0,0)--(2,2);
     \draw[red,ultra thick](4,2)--(5,1);  
     \draw[red,ultra thick](9,1)--(10,2); 
     \draw[red,ultra thick](12,2)--(14,0);     
    \end{tikzpicture}
   \caption{A grand Dyck path and a Dyck path, both with degree of symmetry 3. The symmetric steps and their mirror images are highlighted in red.}\label{fig:symD}
\end{figure}
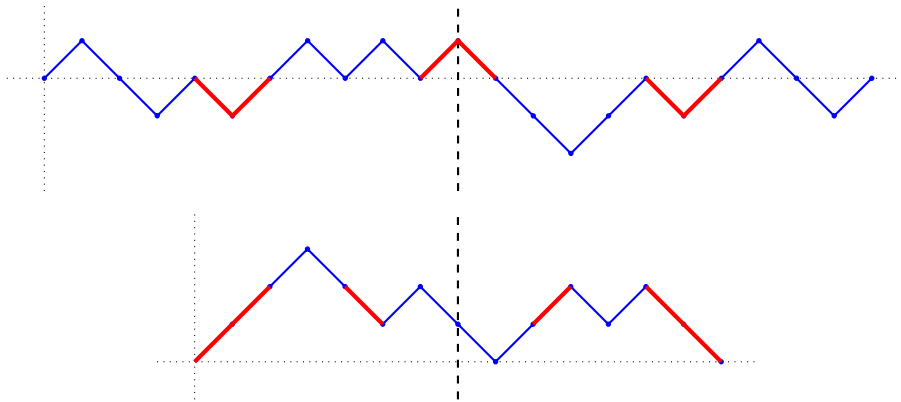

In Section~\ref{sec:preliminaries} we give some background on generalized Motzkin paths, along with some enumerative results on these paths and tools to determine limiting distributions that will be used in later sections. In Section~\ref{sec:GD} we derive an expression for the generating function for grand Dyck paths with respect to semilength and degree of symmetry. 
Section~\ref{sec:partitions} focuses on the degree of symmetry of partitions, giving generating functions with respect to different definitions of size and degree of symmetry. In Section~\ref{sec:unimodal} we consider unimodal compositions whose maximum is in the middle, and we enumerate them with respect to semiperimeter (by viewing compositions as bargraphs) and degree of symmetry.
Section~\ref{sec:Dyck} deals with the enumeration of Dyck paths by the degree of symmetry. Using bijections for walks in the quarter plane, we derive a functional equation for the corresponding generating function.
Surpisingly, the behavior of the degree of symmetry on Dyck paths is significantly more intricate than on grand Dyck paths. In Section~\ref{sec:other}, we try to understand this phenomenon by looking at a related statistic on these paths. 
Our research provides connections to several sequences in the OEIS~\cite{OEIS}, which are summarized in Table~\ref{tab:OEIS}.

An extended abstract of this work, which was presented at the conference FPSAC'20, appears in~\cite{EliFPSAC}.

\begin{table}[htb]
\centering
\resizebox{\columnwidth}{!}{
\begin{tabular}{c|c|l}
OEIS entry~\cite{OEIS} & location in paper & combinatorial intepretation \\ \hline\hline
{\bf A341415} & Thm.~\ref{thm:sym_G} & grand Dyck paths by degree of symmetry\\ \hline
A108747 & Thm.~\ref{thm:sv_G}, Cor.~\ref{cor:sym_P} & grand Dyck paths by number of symmetric vertices;\\ && grand Dyck paths by number of returns\\ \hline
{\bf A341445}, A298645 & Thm.~\ref{thm:symDyck}, $D(s,z)$ & Dyck paths by degree of symmetry\\ \hline
{\bf A339754} & Thm.~\ref{thm:svDyck}, $\vD(v,z)$ & Dyck paths by number of symmetric vertices\\ \hline
A051286 & Eq.~\eqref{eq:UB}, Thm.~\ref{thm:Theta} & unimodal bargraphs with centered maximum;\\ && grand Motzkin paths with no peaks;\\ && uneven bicolored grand Motzkin paths \\ \hline
A004148 & Eq.~\eqref{eq:peaklessMotzkin}, Prop.~\ref{prop:bijections} & Motzkin paths with no peaks;\\ && Motzkin paths with no valleys; \\ && uneven bicolored Motzkin paths \\ \hline
A005817 & Sec.~\ref{sec:Dyck} & walks in first quadrant with steps 
\begin{tikzpicture}[scale=.2]
\draw[->] (0,0)--(1,0);
\draw[->] (0,0)--(1,-1);
\draw[->] (0,0)--(-1,1);
\draw[->] (0,0)--(-1,0);
\end{tikzpicture}
ending on $x$-axis \\ \hline
A005558 & Sec.~\ref{sec:Dyck} & walks in first quadrant with steps \begin{tikzpicture}[scale=.2]
\draw[->] (0,0)--(1,0);
\draw[->] (0,0)--(1,-1);
\draw[->] (0,0)--(-1,1);
\draw[->] (0,0)--(-1,0);
\end{tikzpicture} and arbitrary endpoint \\ \hline
A001246 & Thm.~\ref{thm:symDyck}, $R(0,0,1,z)$ & Dyck paths with midpoint at height $0$ \\ \hline
A213600 & Thm.~\ref{thm:symDyck}, $R(x,0,1,z)$ & Dyck paths by height of midpoint \\ \hline
A018224 & Thm.~\ref{thm:symDyck}, $R(1,1,1,z)$ & Dyck paths allowed to be discontinuous (i.e., can jump to a \\
&& different height) at midpoint
\\ \hline
A000891 & Thm.~\ref{thm:symDyck}, $R(0,1,1,z)$ & Dyck paths allowed to be discontinuous at midpoint,\\
&& with one half being a Dyck path
\end{tabular}
}
\caption{A summary of the OEIS sequences the appear in the paper. The entries in boldface indicate new additions to the database~\cite{OEIS}.}
\label{tab:OEIS}
\end{table}

\section{Preliminaries}\label{sec:preliminaries}

In Section~\ref{sec:bicolored} we present some facts about bicolored Motzkin paths that will be used in the proofs of our results in Sections~\ref{sec:GD}, \ref{sec:partitions}, and~\ref{sec:unimodal}.
In Section~\ref{sec:limiting} we introduce some tools that will be needed to describe the limiting distribution of the degree of symmetry and related statistics on various objects, once we find the corresponding bivariate generating functions. 

\subsection{Bicolored grand Motzkin paths}\label{sec:bicolored}

Like a grand Dyck path, a {\em bicolored grand Motzkin path} starts at the origin and ends on the $x$-axis, but, in addition to steps $U=(1,1)$ and $D=(1,-1)$, it may contain horizontal steps $(1,0)$ of two kinds (or colors), denoted by $H_1$ and $H_2$. 
If such a path does not go below the $x$-axis, it is called a {\em bicolored Motzkin path}.
Denote by $\GM^2$ the set of bicolored grand Motzkin paths, and by $\M^2$ the set of bicolored Motzkin paths. We often identify a path with its sequence of steps.

For a path $M\in\GM^2$, let $u(M)$ denote its number of $U$ steps (which also equals its number $d(M)$ of $D$ steps), and define $h_1(M)$ and $h_2(M)$ analogously.
Additionally, let $h_1^0(M)$ denote the number of $H_1$ steps of $M$ on the $x$-axis, and define $h_2^0(M)$ similarly.
Define the {\em length} of $M$ to be its total number of steps, which we denote by $|M|$. Note that $|M|=u(M)+d(M)+h_1(M)+h_2(M)$. Let $\M^2_n\subset\M^2$ and $\GM^2_n\subset\GM^2$ denote the subsets consisting of paths of length $n$ in each case.

\begin{lemma}\label{lem:2Motzkin}
Let $F(x,y)=\sum_{M\in\M^2} x^{d(M)+h_1(M)} y^{u(M)+h_2(M)}$. Then
$$F(x,y)=\frac{1-x-y-\sqrt{(1-x-y)^2-4xy}}{2xy}.$$
\end{lemma}

\begin{proof}
A non-empty path $M\in\M^2$ can be decomposed uniquely as $H_1M'$, $H_2M'$ or $UM'DM''$, where $M',M''\in\M^2$ (translated appropriately). This decomposition yields the following equation for the generating function:
$$F(x,y)=1+(x+y)F(x,y)+xyF(x,y)^2.$$
Solving for $F(x,y)$ and taking the sign of the square root that results in a formal power series, we get the stated expression.
\end{proof}

\begin{lemma}\label{lem:grand2Motzkin}
Let $G(x,y,s_1,s_2)=\sum_{M\in\GM^2} x^{d(M)+h_1(M)} y^{u(M)+h_2(M)} s_1^{h_1^0(M)} s_2^{h_2^0(M)}$. Then
$$G(x,y,s_1,s_2)=\frac{1}{(1-s_1)x+(1-s_2)y+\sqrt{(1-x-y)^2-4xy}}.$$
\end{lemma}

\begin{proof}
For $M'\in\M^2$, denote by $\overline{M'}$ the path in $\GM^2$ obtained by reflecting $M'$ along the $x$-axis; equivalently, by changing the $U$ steps into $D$ steps and vice versa. 

A path $M\in\GM^2$ can be written uniquely as a sequence of steps $H_1$ on the $x$-axis, steps $H_2$ on the $x$-axis,  paths of the form $UM'D$, and paths of the form $D\overline{M'}U$, where $M'\in\M^2$. It follows that
$$G(x,y,s_1,s_2)=\frac{1}{1-s_1x-s_2y-2xyF(x,y)}.$$
Using the expression for $F(x,y)$ given by Lemma~\ref{lem:2Motzkin}, we obtain the desired formula.
\end{proof}

\subsection{Limiting distributions}\label{sec:limiting}

Let $A(u,z)=\sum_{k,n}a_{k,n} u^kz^n$ be a bivariate generating function enumerating a class of objects with respect to their size and some statistic, so that $a_{k,n}$ is the number of objects of size $n$ where the statistic equals $k$. Let $a_n=\sum_k a_{k,n}$ be the number of objects of size $n$. To study the distribution of the statistic on objects of a given size, we define a sequence of random variables $X_n$ by
\begin{equation}\label{eq:X_n} \Pr(X_n=k)=\frac{a_{k,n}}{a_n}.
\end{equation}
 
Denote by $A_u$ and $A_{uu}$ the first and second partial derivatives of $A$ with respect to $u$, respectively. It is well known that the expected value and the variance of $X_n$ are given by
\begin{equation}\label{eq:EV}
\E X_n=\frac{[z^n]A_u(1,z)}{a_n} \quad\text{and}\quad
\V X_n=\frac{[z^n]A_{uu}(1,z)+[z^n]A_u(1,z)}{a_n}-\left(\frac{[z^n]A_u(1,z)}{a_n}\right)^2,
\end{equation}
which can be easily computed if we know $A(u,z)$. The following two theorems describe the limiting distribution of $X_n$ when $A(u,z)$ has a certain form. They are due to Drmota and Soria~\cite[Thm.\ 1]{DrmSor} and Wallner~\cite[Thm.\ 2.1]{Wallner}, respectively.

\begin{theorem}[\cite{DrmSor}]\label{thm:Rayleigh}
Suppose that the following hold: 
\begin{enumerate}[(i)]
\item $A(u,z)$ is a bivariate power series with nonnegative coefficients;
\item $A(1,z)$ has radius of convergence $\rho>0$;
\item there exist analytic functions $g(u,z)$ and $h(u,z)$ such that  
$A(u,z)^{-1}=g(u,z)+h(u,z)\sqrt{1-z/\rho}$
for $|u-1|<\epsilon$ and $|z-\rho|<\epsilon$, $\arg(z-\rho)\neq0$, for some $\epsilon>0$;
\item $g(1,\rho)=0$;
\item $z=\rho$ is the only singularity on the circle $|z|=\rho$;
\item $A(u,z)$ can be analytically continued to a region $|z|<\rho+\delta$, $|u|<1+\delta$, $|u-1|>\epsilon/2$, and $\arg(z-\rho)\neq0$ for some $\delta>0$;
\item $h(1,\rho)>0$ and $g_u(1,\rho)<0$.
\end{enumerate}
Then the sequence of random variables~\eqref{eq:X_n} has a Rayleigh limiting distribution, that is,
$$\frac{X_n}{\sqrt{n}}\stackrel{d}{\to}\Ral(\sigma), \qquad \text{where}\quad \sigma=\frac{\sqrt{2}\,|g_u(1,\rho)|}{h(1,\rho)},$$ and $\Ral(\sigma)$ has density $\frac{x}{\sigma^2}e^{-x^2/2\sigma^2}$ for $x\ge0$. Expected value and variance are given by
$$\E X_n=\sigma\sqrt{\frac{\pi n}{2}}+O(1)\quad\text{and}\quad \V X_n=\sigma^2\left(2-\frac{\pi}{2}\right)n+O(\sqrt{n}),$$
and we have the local law
$$\Pr(X_n=k)=\frac{k}{\sigma^2 n}e^{-k^2/2\sigma^2n}+O((k+1)n^{-3/2})+O(n^{-1})$$
uniformly for all $k\ge0$.
\end{theorem}

\begin{theorem}[\cite{Wallner}]\label{thm:half-normal}
Suppose that conditions (i)--(vi) from Theorem~\ref{thm:Rayleigh} hold, and also the following:
\begin{enumerate}[(vii')]
\item $h(1,\rho)=g_u(1,\rho)=g_{uu}(1,\rho)=0$, \ $g_z(1,\rho)\neq0$ and $h_u(1,\rho)\neq0$.
\end{enumerate}
Then the sequence of random variables~\eqref{eq:X_n} has a half-normal limiting distribution, that is,
$$\frac{X_n}{\sqrt{n}}\stackrel{d}{\to}\Hn(\sigma), \qquad \text{where}\quad \sigma=\frac{\sqrt{2}\,|h_u(1,\rho)|}{\rho\,|g_z(1,\rho)|},$$ and $\Hn(\sigma)$ has density $\frac{1}{\sigma}\sqrt{\frac{2}{\pi}}\,e^{-x^2/2\sigma^2}$  for $x\ge0$. Expected value and variance are given by
$$\E X_n=\sigma\sqrt{\frac{2}{\pi n}}+O(1)\quad\text{and}\quad \V X_n=\sigma^2\left(1-\frac{2}{\pi}\right)n+O(\sqrt{n}),$$
and we have the local law
$$\Pr(X_n=k)=\frac{1}{\sigma}\sqrt\frac{2}{\pi n}e^{-k^2/2\sigma^2n}+O(kn^{-3/2})+O(n^{-1})$$
uniformly for all $0\le k\le k_0\sqrt{n\log n}$ with $k_0<\sigma$.
\end{theorem}

The half-normal distribution is obtained when taking the absolute variable of a normally distributed random variable with zero mean.

\section{Symmetry of grand Dyck paths}\label{sec:GD}

The main result of this section is the following remarkably simple formula. 

\begin{theorem} \label{thm:sym_G} The generating function for grand Dyck paths with respect to their degree of symmetry is
$$\sum_{n\ge0} \sum_{P\in\G_n} s^{\ds(P)} z^n = \frac{1}{2(1-s)z + \sqrt{1-4z}}.$$
\end{theorem}

Before we prove this theorem, let us introduce some notation.
Given $P\in\G_n$, one can construct two paths as follows. Let $P_L$ and $P_R$ denote the left half and the right half of $P$, respectively. Let $P'_R$ be the path obtained by reflecting $P_R$ along the vertical line $x=n$. Note that $P_L$ and $P'_R$ are paths with steps $U$ and $D$ from $(0,0)$ to some common endpoint on the line $x=n$. Denote the $i$th step of $P_L$ by $\bar{\ell}_i$ when viewed as a segment in the plane, and let $\ell_i\in\{U,D\}$ be the direction of this step. Define $\bar{r}_i$ and $r_i$ similarly for the path $P'_R$. Next we describe a bijection $\bij$ from $\G_n$ to $\GM^2_n$; see
Figure~\ref{fig:pairs2Motzkin} for an example. 

\begin{definition}\label{def:bij}
For $P\in\G_n$ with the above notation, let $\bij(P)\in\GM^2_n$ be the path whose $i$th step is equal to 
$$\begin{cases} U & \text{ if } \ell_i=U \text{ and }r_i=D,\\
D & \text{ if } \ell_i=D \text{ and }r_i=U,\\
H_1 & \text{ if } \ell_i=r_i=D,\\
H_2 & \text{ if } \ell_i=r_i=U.
\end{cases}$$
\end{definition}

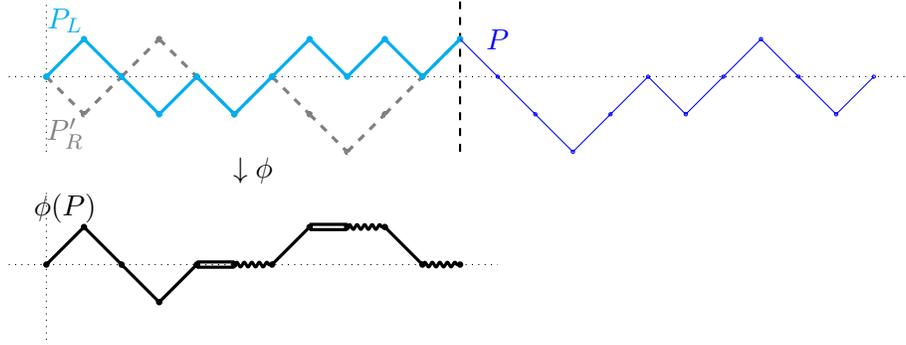
\begin{figure}[htb]
\centering
    \begin{tikzpicture}[scale=0.5]
     \draw[dotted](-1,0)--(23,0);
     \draw[dotted](0,-2)--(0,2);
     \draw[dashed,thick](11,-2)--(11,2);
    \draw[gray,very thick,dashed](0,0) circle (1.2pt)  \D\U\U\D\D\U\D\D\U\U\U;
    \draw[cyan,very thick](0,0) circle (1.2pt) \U\D\D\U\D\U\U\D\U\D\U coordinate(last);
    \draw[blue] (last) \D\D\D\U\U\D\U\U\D\D\U;
    \draw[gray] (.5,-1.5) node {$P'_R$};
    \draw[cyan] (.5,1.5) node {$P_L$};
    \draw[blue] (12,1) node {$P$};
    \draw (5.5,-2.5) node  {$\downarrow\bij$};
    \begin{scope}[shift={(0,-5)}]
      \draw[dotted](-1,0)--(12,0);
     \draw[dotted](0,-2)--(0,2);
    \draw[very thick](0,0) circle (1.2pt)  \U\D\D\U\Hone coordinate(last);
     \drawHtwo
    \draw[very thick](last) circle (1.2pt) \U\Hone coordinate(last);
    \drawHtwo 
    \draw[very thick](last) circle (1.2pt) \D coordinate(last);
    \drawHtwo
    \draw[very thick](last) circle (1.2pt);
    \draw (.5,1.5) node {$\bij(P)$};
    \end{scope}
    \end{tikzpicture}
   \caption{An example of the bijection $\bij:\G_n\to\GM^2_n$ from Definition~\ref{def:bij}. Above, the path $P\in\G_n$ from the top of Figure~\ref{fig:symD}
   is drawn in blue, with its left half $P_L$ drawn in light blue and its reflected right half $P'_R$  drawn in gray with dashes. Below is the path $\bij(P)\in\GM^2$, whose $H_1$ and $H_2$ steps are pictured with double lines and wavy lines, respectively.}\label{fig:pairs2Motzkin}
\end{figure}

\begin{lemma}\label{lem:bij_properties}
The map $\bij:\G_n\to\GM^2_n$ is a bijection with the property that, if $M=\bij(P)$, then $\ds(P)=h_1^0(M)+h_2^0(M)$.
\end{lemma}

\begin{proof}
A path $P\in\G_n$ is symmetric in position $i$ if and only if the steps $\bar{\ell}_i$ and $\bar{r}_i$ coincide as segments. Thus, the degree of symmetry of $P$ equals the number of common (i.e.\ overlapping) steps of $P_L$ and $P'_R$.

By construction of $\bij$, the height (i.e., $y$-coordinate) of each vertex of $M$ is half of the difference of heights of the corresponding vertex of $P'_R$ and the corresponding vertex of $P_L$. Thus, steps where $P_L$ and $P'_R$ coincide become horizontal steps of $M$ at height~$0$. 
\end{proof}

\begin{proof}[Proof of Theorem~\ref{thm:sym_G}]
Combining Lemmas~\ref{lem:bij_properties} and~\ref{lem:grand2Motzkin},
\[
\sum_{n\ge0} \sum_{P\in\G_n} s^{\ds(P)} z^n  = \sum_{M\in\GM^2} s^{h_1^0(M)+h_2^0(M)} z^{|M|} = G(z,z,s,s)=\frac{1}{2(1-s)z+\sqrt{1-4z}}.\qedhere
\]
\end{proof}

When $P_L$ lies strictly above $P'_R$ (except at their common endpoints), the pair $(P_L,P'_R)$ is called a {\em parallelogram polyomino} \cite{Pol,Sha,DV}, and its semiperimeter is defined to be the length of either of the two paths. The bijection $\bij$ in Definition~\ref{def:bij}, after removing the first and the last step of the image path, restricts to a bijection between parallelogram polyominos of semiperimeter $n$ and bicolored Motzkin paths of length $n-2$, which are known to be enumerated by the Catalan number $C_{n-1}$. Indeed, using this bijection and Lemma~\ref{lem:2Motzkin}, the generating function for parallelogram polyominoes where $z$ marks the semiperimeter is $$z^2F(z,z)=\frac{1-2z-\sqrt{1-4z}}{2}=z(\Cat(z)-1),$$
where 
\begin{equation}\label{eq:Cat} \Cat(z)=\sum_{n\ge0} C_n z^n = \frac{1-\sqrt{1-4z}}{2z} 
\end{equation} 
is the generating function for the Catalan numbers. 
More generally, $xyF(x,y)$ is the generating function for parallelogram polyominoes where $x$ and $y$ mark the number of $D$ and $U$ steps, respectively, of the upper (equivalently, the lower) path.

When $P_L$ lies weakly above $P'_R$, the pair $(P_L,P'_R)$ is sometimes called a {\em $2$-watermelon} \cite{Fisher,GOV,Roitner}. Such objects are again enumerated by the Catalan numbers, since they are obtained from parallelogram polyominos by removing the first and the last steps of both paths. The distribution of the number of common vertices and common steps in $2$-watermelons has independently been studied by Roitner~\cite{Roitner}, using a restriction of the bijection $\bij$ to this special case.

Applying the tools from Section~\ref{sec:limiting} to Theorem~\ref{thm:sym_G}, we can describe the limiting distribution of the degree of symmetry of grand Dyck paths.

\begin{proposition}\label{prop:limit-ds}
The sequence of random variables $X_n$ giving the degree of symmetry of a uniformly random grand Dyck path in $\G_n$ has a Rayleigh limiting distribution. Specifically,
$$\frac{X_n}{\sqrt{n}}\stackrel{d}{\to}\Ral\left(\frac{1}{\sqrt{2}}\right),$$ 
which has density function $2xe^{-x^2}$ for $x\ge0$, and
$$\E X_n=\frac{\sqrt{\pi n}}{2}+O(1),\qquad \V X_n=\left(1-\frac{\pi}{4}\right)n+O(\sqrt{n}),$$
$$\Pr(X_n=k)=\frac{2k}{n}e^{-k^2/n}+O((k+1)n^{-3/2})+O(n^{-1})$$
uniformly for all $k\ge0$.
\end{proposition}

\begin{proof}
We apply Theorem~\ref{thm:Rayleigh} to the bivariate generating function from Theorem~\ref{thm:sym_G} (using the variable $s$ instead of $u$), with $\rho=1/4$, $g(s,z)=2(1-s)z$, and $h(s,z)=1$. The resulting Rayleigh distribution has parameter
$$\sigma=\frac{\sqrt{2}\,|g_s(1,1/4)|}{h(1,1/4)}=\frac{1}{\sqrt{2}},$$
from where the rest of the statements follow.

Note that the exact expected value and the variance can also be computed directly from Equation~\eqref{eq:EV}:
$$\E X_n=\frac{2^{2n-1}}{\binom{2n}{n}},\qquad
\V X_n
=\frac{2n(n-1)}{2n-1}+\frac{2^{2n-1}}{\binom{2n}{n}}-\frac{4^{2n-1}}{\binom{2n}{n}^2}.
$$
\end{proof}

\medskip

An alternative measure of the symmetry of a grand Dyck path is its {\em number of symmetric vertices}, that is, vertices in the first half of the path that are mirror images of vertices in the second half, again with respect to reflection along the vertical line that passes through the midpoint of the path. We do not consider the midpoint of the path as a symmetric vertex. For example, the grand Dyck path in Figure~\ref{fig:sv} has 6 symmetric vertices. Denote by $\sv(P)$ the number of symmetric vertices of the path $P\in\G_n$. The following is the analogue of Theorem~\ref{thm:sym_G} when considering symmetric vertices instead of symmetric steps.

\begin{figure}[htb]
\centering
   \begin{tikzpicture}[scale=0.5] 
     \draw[dotted](-1,0)--(23,0);
     \draw[dotted](0,-3)--(0,2);
     \draw[dashed,thick](11,-3)--(11,2);
     \draw[thick,blue](0,0) circle(1.2pt) \U\D\D\U\D\U\U\D\U\D\U  \D\D\D\U\U\D\U\U\D\D\U;
     \foreach \x in {(0,0),(2,0),(4,0),(5,-1),(6,0),(10,0)}{ \draw[red,thick] \x circle(.15);}  
     \foreach \x in {(22,0),(20,0),(18,0),(17,-1),(16,0),(12,0)}{ \draw[red,thick] \x circle(.15);}  
    \end{tikzpicture}
   \caption{A grand Dyck path with 6 symmetric vertices, highlighted in red along with their mirror images.}\label{fig:sv}
\end{figure}
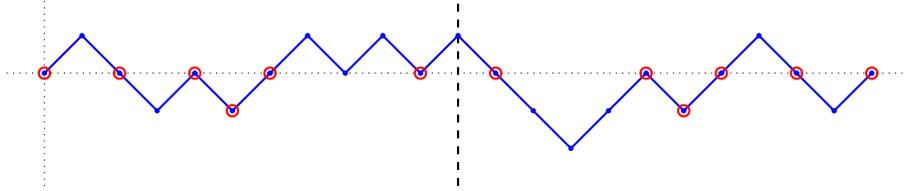

\begin{theorem} \label{thm:sv_G} The generating function for grand Dyck paths with respect to their number of symmetric vertices is
$$\sum_{n\ge0} \sum_{P\in\G_n} v^{\sv(P)} z^n = \frac{1}{1-2vzC(z)}=\frac{1}{1-v + v\sqrt{1-4z}}.$$
\end{theorem}

\begin{proof}
Let $P\in\G_n$ and $M=\bij(P)$, where $\bij$ is the bijection from Definition~\ref{def:bij}. Then $\sv(P)$ equals the number of vertices of $M$ on the $x$-axis minus one. As in the proof of Lemma~\ref{lem:grand2Motzkin},
we can decompose $M\in\GM^2$ uniquely as a sequence of steps $H_1$ and $H_2$ on the $x$-axis, paths of the form $UM'D$, and paths of the form $D\overline{M'}U$, where $M'\in\M^2$. Since each one of these four types of blocks contributes one new symmetric vertex, it follows that
$$\sum_{n\ge0} \sum_{P\in\G_n} v^{\sv(P)} z^n = \frac{1}{1-2vz-2vz^2F(z,z)}=\frac{1}{1-2vzC(z)},$$
with $F$ given by Lemma~\ref{lem:2Motzkin} and $C(z)$ given by Equation~\eqref{eq:Cat}.
\end{proof}

The number of symmetric vertices also follows a Rayleigh limit law.

\begin{proposition}
The sequence of random variables $X_n$ giving the number of symmetric vertices of a uniformly random grand Dyck path in $\G_n$ has a Rayleigh limiting distribution. Specifically,
$$\frac{X_n}{\sqrt{n}}\stackrel{d}{\to}\Ral(\sqrt{2}),$$ 
which has density function $\frac{x}{2}e^{-x^2/4}$ for $x\ge0$, and
$$\E X_n=\sqrt{\pi n}+O(1),\qquad \V X_n=(4-\pi)n+O(\sqrt{n}),$$
$$\Pr(X_n=k)=\frac{k}{2n}e^{-k^2/4n}+O((k+1)n^{-3/2})+O(n^{-1})$$
uniformly for all $k\ge0$.
\end{proposition}

\begin{proof}
Applying Theorem~\ref{thm:Rayleigh} to the bivariate generating function from Theorem~\ref{thm:sv_G}, we have $\rho=1/4$, $g(v,z)=1-v$, and $h(v,z)=v$, from where
$$\sigma=\frac{\sqrt{2}\,|g_v(1,1/4)|}{h(1,1/4)}=\sqrt{2}.$$

The exact expected value and the variance can also be computed directly from Equation~\eqref{eq:EV}:
\[
\E X_n=\frac{4^{n}}{\binom{2n}{n}}-1,\qquad
\V X_n=4n+2-\frac{4^{n}}{\binom{2n}{n}}-\frac{16^{n}}{\binom{2n}{n}^2}.
\qedhere
\]
\end{proof}

For $P\in\G_n$, denote by $\ret(P)$ the number of returns of $P$ to the $x$-axis, that is, the number of steps that end on the $x$-axis.

\begin{corollary}
The statistics $\sv$ and $\ret$ are equidistributed on $\G_n$; that is, for all $n,k\ge0$,
$$|\{P\in\G_n:\sv(P)=k\}|=|\{P\in\G_n:\ret(P)=k\}|.$$
\end{corollary}

\begin{proof}
It is easy to see directly that the generating function in Theorem~\ref{thm:sv_G} enumerates grand Dyck paths with respect to the number of returns to the $x$-axis~\cite[A108747]{OEIS}, from where the result follows.

Alternatively, we can give the following bijective proof of this equality. Given $P\in\G_n$, first apply $\bij$ and let $M=\bij(P)\in\GM_n^2$. Then replace each step $U$ of $M$ with $UU$, each step $D$ with $DD$, each step $H_1$ with $UD$, and each step $H_2$ with $DU$. Let $Q$ be the resulting grand Dyck path. This map $P\mapsto Q$ is a bijection from $\G_n$ to itself with the property that $\sv(P)=\ret(Q)$. As an example, the image of the path in Figure~\ref{fig:sv} is given in Figure~\ref{fig:returns}.
\end{proof}

\begin{figure}[htb]
\centering
   \begin{tikzpicture}[scale=0.5] 
     \draw[dotted](-1,0)--(23,0);
    \draw[dotted](0,-2)--(0,3);
     \draw[thick,blue](0,0) circle(1.2pt) \U\U\D\D\D\D\U\U\U\D\D\U\U\U\U\D\D\U\D\D\D\U;
     \foreach \x in {4,8,10,12,20,22}{ \draw[red,thick] (\x,0) circle(.15);}  
    \end{tikzpicture}
   \caption{A grand Dyck path with 6 returns to the $x$-axis, highlighted in red.}\label{fig:returns}
\end{figure}

The first few coefficients of the generating functions from Theorems~\ref{thm:sym_G} and~\ref{thm:sv_G} are given in Table~\ref{tab1}.

\begin{table}[h]
\centering
\begin{tabular}{c@{\quad\qquad}c}
$\card{\{P \in \G_n: \ds(P) = k\}}$ & $\card{\{P \in \G_n: \sv(P) = k\}}$ \smallskip \\
\begin{tabular}{c||r|r|r|r|r|r|r|r|} 
 $n \setminus k$ & 0& 1 & 2 & 3 & 4 & 5 & 6\\ 
\hline\hline
1& 0&2 & & & & &   \\
\hline
2&2 &0 &4 & & & & \\
\hline
3&4 &8 &0 & 8 & & & \\
\hline
4&14 &16 &24 & 0 & 16 & & \\
\hline
5& 44 & 64 & 48 & 64 & 0 & 32 & \\
\hline
6& 148 & 208 & 216 & 128 & 160 & 0 & 64 \\
\hline
\end{tabular}
&
\begin{tabular}{c||r|r|r|r|r|r|r|} 
 $n \setminus k$ & 1 & 2 & 3 & 4 & 5 & 6\\ 
\hline\hline
1& 2 & & & & &   \\
\hline
2&2 &4 & & & & \\
\hline
3&4 &8 & 8 & & & \\
\hline
4&10 &20 &24 & 16 & & \\
\hline
5& 28 & 56 & 72 & 64 & 32 & \\
\hline
6& 84 & 168 & 224 & 224 & 160 & 64 \\
\hline
\end{tabular}
\end{tabular}
\caption{The number of grand Dyck paths of length $n\le 6$ with a given degree of symmetry (left, see Theorem~\ref{thm:sym_G}) and with a given number of symmetric vertices (right, see Theorem~\ref{thm:sv_G}).}
\label{tab1}
\end{table}

To end this section, we point out that the method used to prove Theorems~\ref{thm:sym_G} and~\ref{thm:sv_G} also gives the generating function that keeps track of both statistics simultaneously:
$$\sum_{n\ge0}\sum_{P\in\G_n}s^{\ds(P)}v^{\sv(P)}z^n=\frac{1}{1-v+2(1-s)vz+v\sqrt{1-4z}}.$$

\section{Symmetry of partitions}\label{sec:partitions}

Let $\P$ denote the set of integer partitions, that is, sequences $\lambda=(\lambda_1,\lambda_2,\dots,\lambda_k)$ where $k\ge0$ and $\lambda_1\ge\lambda_2\ge\dots\ge\lambda_k\ge1$.  The entries $\lambda_i$ are called {\em parts}.
We draw the Young diagram of $\lambda$ in English notation, by arranging boxes (unit squares) into $k$ left-justified rows, where the $i$th row from the top has $\lambda_i$ boxes for each $i$; see Figure~\ref{fig:Young} for an example.
The {\em conjugate} of $\lambda$, denoted by $\lambda'$, is the partition with parts $\lambda'_i=|\{j:\lambda_j\ge i\}|$ for $1\le i\le \lambda_1$. 
Note that $\lambda'_1$ equals the number of parts of~$\lambda$, and that the Young diagram of $\lambda'$ is obtained by transposing the Young diagram of $\lambda$.

\begin{figure}[htb]
\centering
    \begin{tikzpicture}[scale=0.5]
      \draw (0,-6) grid (1,0);
      \draw (1,-4) grid (2,0);
      \draw (2,-3) grid (4,0);
      \draw (4,-1) grid (5,0);
 \begin{scope}[shift={(9,0)}]
      \draw (6,0) grid (0,-1);
      \draw (4,-1) grid (0,-2);
      \draw (3,-2) grid (0,-4);
      \draw (1,-4) grid (0,-5);
      \end{scope}
   \end{tikzpicture}
  \caption{The Young diagrams of $\lambda=(5,4,4,2,1,1)$ and its conjugate $\lambda'=(6,4,3,3,1)$.}\label{fig:Young}
\end{figure}
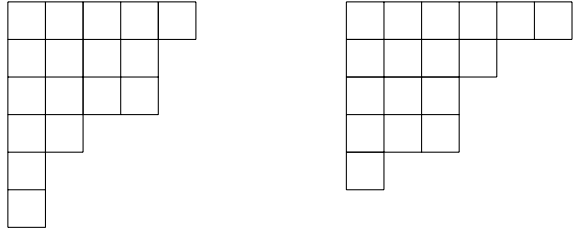

Each one of the following three subsections considers a different measure of symmetry of a partition. The first one views partitions inside a square and relates them to grand Dyck paths. The second one is perhaps the most natural measure of symmetry: the number of parts that equal the corresponding part in the conjugate partition. The third measure involves a decomposition of partitions into diagonal hooks.

\subsection{Partitions inside a square}

Let $\P^\Box_n\subset\P$ be the set of partitions $\lambda=(\lambda_1,\lambda_2,\dots,\lambda_k)$ with $k\le n$ and $\lambda_1\le n$. These can be thought of as partitions whose Young diagram fits inside an $n\times n$ square. 
For $\lambda\in \P^\Box_n$, let $\tilde\lambda=(\lambda_1,\lambda_2,\dots,\lambda_k,0,\dots,0)$ denote the sequence of length $n$ obtained by appending $n-k$ zeros to $\lambda$. 
Let $\tilde\lambda'$ be the sequence of length $n$ obtained by conjugating $\tilde\lambda$, that is, having entries $\tilde\lambda'_i=|\{j:\tilde\lambda_j\ge i\}|$ for $1\le i\le n$. 

Viewing $\lambda$ as a partition inside a square, one can define the following measure of symmetry: 
$$\ds^\Box_n(\lambda)=|\{i\in [n]:\tilde\lambda_i=\tilde\lambda'_i\}|.$$ 
Note that zeros in $\tilde\lambda$ may contribute to $\ds^\Box_n(\lambda)$.
For example, if $\lambda=(5,4,4,2,1,1)$,
then $\ds^\Box_6(\lambda)=2$ but $\ds^\Box_7(\lambda)=3$, since in the second case, $\tilde\lambda=(5,4,4,2,1,1,0)$ and $\tilde\lambda'=(6,4,3,3,1,0,0)$ coincide in positions $2$, $5$, and $7$.

To relate partitions inside a square and grand Dyck paths, we define a straightforward bijection $\partial_n:\P^\Box_n\to\G_n$ as follows.
\begin{definition}\label{def:partial}
For $\lambda\in\P^\Box_n$, let
$$\partial_n(\lambda)=D^{\tilde\lambda_n}UD^{\tilde\lambda_{n-1}-\tilde\lambda_n}UD^{\tilde\lambda_{n-2}-\tilde\lambda_{n-1}}U\dots D^{\tilde\lambda_{1}-\tilde\lambda_{2}}UD^{n-\tilde\lambda_{1}}\in\G_n.$$
\end{definition}

This bijection can be visualized by placing the Young diagram of $\tilde\lambda$ inside an $n\times n$ square (aligned with the top and left edges), reading the south-east boundary of the diagram from the south-west corner of the square to the north-east corner, and then translating north steps to $U$ steps and east steps to $D$ steps. An example is given in Figure~\ref{fig:partialn}. Under this bijection, which essentially rotates the boundary of the Young diagram by $45^\circ$, conjugating the Young diagram corresponds to reflecting the grand Dyck path along a vertical line. The following is a consequence of Theorem~\ref{thm:sym_G}.

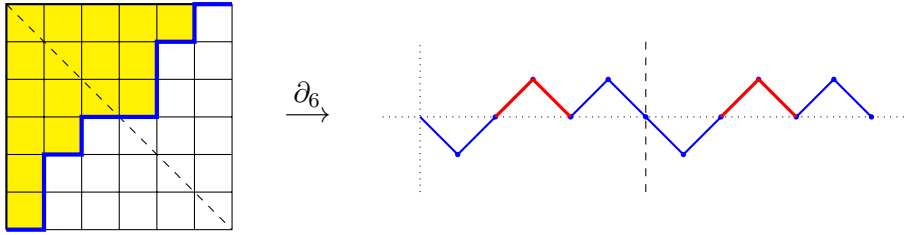
\begin{figure}[htb]
\centering
    \begin{tikzpicture}[scale=0.5]
     \draw[thick,fill=yellow] (0,-6)--(1,-6)--(1,-4)--(2,-4)--(2,-3)--(4,-3)--(4,-1)--(5,-1)--(5,0)--(0,0)--(0,-6);
     \draw[dashed](0,0)--(6,-6);
     \draw[thin] (0,-6) grid (6,0);
     \draw[ultra thick,blue] (0,-6)--(1,-6)--(1,-4)--(2,-4)--(2,-3)--(4,-3)--(4,-1)--(5,-1)--(5,0)--(6,0);
     \draw (8,-3) node  {$\longrightarrow$};
     \draw (8,-3) node[above] {$\partial_6$};
     \begin{scope}[shift={(11,-3)}]
     \draw[dotted](-1,0)--(13,0);
     \draw[dotted](0,-2)--(0,2);
     \draw[dashed](6,-2)--(6,2);
     \draw[thick,blue] (0,0) \D\U\U\D\U\D\D\U\U\D\U\D;
    \draw[red,very thick](2,0)--(3,1)--(4,0);
    \draw[red,very thick](8,0)--(9,1)--(10,0);
     \end{scope}
  \end{tikzpicture}
  \caption{The bijection $\partial_n$ from Definition~\ref{def:partial} for $n=6$ applied to $\lambda=(5,4,4,2,1,1)\in\P^\Box_6$.}\label{fig:partialn}
\end{figure}

\begin{corollary} \label{cor:sym_P} 
The generating function for partitions whose Young diagram fits inside a square with respect to the side length of the square and the statistic $\ds^\Box_n$ is
$$\sum_{n\ge0} \sum_{\lambda\in\P^\Box_n} s^{\ds^\Box_n(\lambda)} z^n = \frac{1}{2(1-s)z + \sqrt{1-4z}}.$$
\end{corollary}

\begin{proof}
Let $\lambda\in\P^\Box_n$, and let $P=\partial_n(\lambda)\in\G_n$ be given by Definition~\ref{def:partial}. Then $\tilde\lambda_i=\tilde\lambda'_i$ if and only if the $i$th $D$ step of $P$ from the left is a mirror image (with respect to the reflection along $x=n$) of its $i$th $U$ step from the right. Thus, 
\begin{equation}\label{eq:dsn}
\ds^\Box_n(\lambda)=\ds(P).
\end{equation} 
The result now follows from Theorem~\ref{thm:sym_G}.
\end{proof}

As a consequence of Corollary~\ref{cor:sym_P}, the limiting distribution of the statistic $\ds^\Box_n$ on partitions in $\P^\Box_n$ coincides with the one described in Proposition~\ref{prop:limit-ds}.

\subsection{Symmetry by self-conjugate parts}

Another natural measure of symmetry of a partition, which we call simply the degree of symmetry of $\lambda\in\P$, is defined as $$\ds(\lambda)=|\{i:\lambda_i=\lambda'_i\}|,$$ 
that is, the number of parts of $\lambda$ that equal the corresponding parts of its conjugate.
For example, if $\lambda=(5,4,4,2,1,1)$, then $\ds(\lambda)=2$ because $\lambda_2=\lambda'_2=4$ and $\lambda_5=\lambda'_5=1$, but $\lambda_i\ne\lambda'_i$ for every other $i$ for which these quantities are defined.

The following straightforward observation relates the two measures of symmetry for partitions defined so far.

\begin{lemma}\label{lem:dsm}
Let $\lambda\in\P$, and let $m=\max\{\lambda_1,\lambda'_1\}$ be the side length of the smallest square where its Young diagram fits.
Then $\ds^\Box_m(\lambda)=\ds(\lambda)$.
\end{lemma}

\begin{proof}
For this choice of $m$, at least one of $\tilde\lambda$ and $\tilde\lambda'$ has no zeros, and so no zero entries contribute to $\ds^\Box_m$.
\end{proof}

In the generating function from Corollary~\ref{cor:sym_P}, each partition $\lambda\in\P$ contributes as a member of $\P^\Box_n$ for every large enough value of $n$. Next we modify this generating function so that each partition is weighted by $s^{\ds(\lambda)}$ and contributes exactly once. 

\begin{corollary}\label{cor:dsmax}
The generating function for partitions with respect to the side length of the smallest square containing their Young diagram and their degree of symmetry is
$$\sum_{\lambda\in\P} s^{\ds(\lambda)} z^{\max\{\lambda_1,\lambda'_1\}} = \frac{1-sz}{2(1-s)z + \sqrt{1-4z}}.$$
\end{corollary}

\begin{proof}
Let $\lambda\in\P$, and let $m=\max\{\lambda_1,\lambda'_1\}$. 
Then $\ds^\Box_m(\lambda)=\ds(\lambda)$ by Lemma~\ref{lem:dsm}. It follows that
\begin{equation}\label{eq:dsdsn}
\sum_{\lambda\in\P} s^{\ds(\lambda)}  z^{\max\{\lambda_1,\lambda'_1\}}
=\sum_{\lambda\in\P} \sum_{n\ge\max\{\lambda_1,\lambda'_1\}} s^{\ds^\Box_n(\lambda)} z^n
-\sum_{\lambda\in\P} \sum_{n>\max\{\lambda_1,\lambda'_1\}} s^{\ds^\Box_n(\lambda)} z^n.
\end{equation}

By letting $j=n-1$ and noting that $\ds^\Box_{j+1}(\lambda)=\ds^\Box_j(\lambda)+1$ for $j\ge \max\{\lambda_1,\lambda'_1\}$, the subtracting term in Equation~\eqref{eq:dsdsn} can be written as
$$\sum_{\lambda\in\P} \sum_{j\ge\max\{\lambda_1,\lambda'_1\}} s^{\ds^\Box_j(\lambda)+1} z^{j+1},$$
and so the right-hand side of Equation~\eqref{eq:dsdsn} equals
$$(1-sz)\sum_{\lambda\in\P} \sum_{n\ge\max\{\lambda_1,\lambda'_1\}} s^{\ds^\Box_n(\lambda)} z^n=
(1-sz)\sum_{n\ge0} \sum_{\lambda\in\P^\Box_n} s^{\ds^\Box_n(\lambda)} z^n,$$
where we used that each partition $\lambda\in\P$ belongs to $\P^\Box_n$ for all $n\ge\max\{\lambda_1,\lambda'_1\}$.
The result now follows from Corollary~\ref{cor:sym_P}.
\end{proof}

The same argument used in Proposition~\ref{prop:limit-ds}, applied now to the generating function in Corollary~\ref{cor:dsmax}, shows that, as $m\to\infty$, the statistic $\ds$ on the set of partitions $\lambda$ with $\max\{\lambda_1,\lambda'_1\}=m$ follows again a Rayleigh limit law  $\Ral(1/\sqrt{2})$, as described in Proposition~\ref{prop:limit-ds}.

For $\lambda\in\P$, let $\sp(\lambda)=\lambda_1+\lambda'_1$ denote the semiperimeter of its Young diagram. Next we count partitions by semiperimeter. 
\begin{theorem}\label{thm:Psp}
The generating function for partitions with respect to their semiperimeter and their degree of symmetry is
\begin{equation}\label{eq:Psp}\sum_{\lambda\in\P} s^{\ds(\lambda)} z^{\sp(\lambda)} =1+\frac{z^2\left(\sqrt{1-4z^2}-(1-s)(1-2z)\right)}{(1-2z)\left(2(1-s)z^2+\sqrt{1-4z^2}\right)}.
\end{equation}
\end{theorem}

\begin{proof}
Let $\lambda\in\P$, and let $m=\max\{\lambda_1,\lambda'_1\}$. Let $P=\partial_m(\lambda)\in\G_m$, given by Definition~\ref{def:partial}, and let $M=\bij(P)$, where $\bij:\G_m\to\GM^2_m$ is the bijection from Definition~\ref{def:bij}. 
Then $\ds(\lambda)=\ds^\Box_m(\lambda)=\ds(P)=h_1^0(M)+h_2^0(M)$, using Lemma~\ref{lem:dsm}, Equation~\eqref{eq:dsn}, and Lemma~\ref{lem:bij_properties}.

Because of the choice of $m$, it is not possible for both $P_L$ and $P'_R$ to begin with a $U$ step.
The semiperimeter $\sp(\lambda)$ equals the combined number of steps of $P_L$ and $P'_R$, not counting any initial $U$ steps before the first $D$ in either of these paths.
To see how this statistic translates to the path $M$, let us consider three cases:

\begin{enumerate}
\item If $\lambda_1=\lambda'_1>0$, then both $P_L$ and $P'_R$ begin with a $D$ step. In this case, $M$ begins with an $H_2$ step, and $\sp(\lambda)$ simply equals twice the length of $M$. The generating function for such paths is $sz^2 G(z^2,z^2,s,s)$, with $G$ given by Lemma~\ref{lem:grand2Motzkin}.

\item If $\lambda_1>\lambda'_1$, then $P'_R$ begins with a $D$ step and $P_L$ begins with a $U$ step.
In this case, $M$ begins with a $U$ step as well, and the $U$ steps in $P_L$ before the first $D$ correspond in $M$ to what we call {\em futile} steps, defined as $U$ and $H_2$ steps before the first $D$ or $H_1$. We want to find the generating function $K(s,z)$ for paths $M\in\GM^2$ that begin with a $U$, where $s$ marks $h_1^0(M)+h_2^0(M)$, and $z$ marks twice the number of steps of $M$, minus the number of futile steps. This is equivalent to defining the weight of $M$ to be the product of its step weights, where each step is assigned weight $z^2$, except for futile steps, which are assigned weight~$z$.

Let $J(z)=F(z^2,z^2)$, with $F$ given by Lemma~\ref{lem:2Motzkin}, be the generating function for $\M^2$ where all steps have weight $z^2$. Let $\tilde{J}(z)$ be the generating function for $\M^2$ where all steps have weight $z^2$ except for futile steps, which have weight $z$.  The usual decomposition of bicolored Motzkin paths into one of $H_1M'$, $H_2M'$ or $UM'DM''$, where $M',M''\in\M^2$, gives
$$\tilde{J}(z)=1+z^2J(z)+z\tilde{J}(z)+z^3\tilde{J}(z)J(z),$$ from where 
\begin{equation}\label{eq:tildeD}
\tilde{J}(z)=\frac{1+z^2J(z)}{1-z-z^3J(z)}=\frac{\sqrt{1-4z^2}-1+2z}{2z(1-2z)}.
\end{equation}

Finally, since every $M\in\GM^2$ that begins with a $U$ can be decomposed uniquely as $U M' D M''$, where $M'\in\M^2$ and $M''\in\GM^2$,
we have that $K(s,z)=z^3 \tilde{J}(z) G(z^2,z^2,s,s)$.

\item If $\lambda_1<\lambda'_1$, then $P_L$ begins with a $D$ step and $P'_R$ begins with a $U$ step. By symmetry, the corresponding generating function is again $K(s,z)=z^3 \tilde{J}(z) G(z^2,z^2,s,s)$.
\end{enumerate}

Combining the above three cases and adding the empty partition yields
$$\sum_{\lambda\in\P} s^{\ds(\lambda)} z^{\sp(\lambda)} = 1+\left(sz^2+2 z^3 \tilde{J}(z)\right) G(z^2,z^2,s,s).$$
Using Equation~\eqref{eq:tildeD} and Lemma~\ref{lem:grand2Motzkin}, we obtain Equation~\eqref{eq:Psp}.
\end{proof}

We note that while the generating function for partitions by semiperimeter is rational, namely $1+\frac{z^2}{1-2z}$ (setting $s=1$ in Equation~\eqref{eq:Psp}), the generating function by semiperimeter and degree of symmetry is not. The first few coefficients of this algebraic generating function, given by Theorem~\ref{thm:Psp}, are shown in Table~\ref{tab:partitions}.

\begin{table}[h]
\centering
\begin{tabular}{c||r|r|r|r|r|r|r|r|} 
 $n \setminus k$ & 0& 1 & 2 & 3 & 4 \\ 
\hline\hline
2&0 &1 & 0 & 0 & 0 \\
\hline
3&2 &0 &0 &0 &0 \\
\hline
4&2 &0 &2 & 0 & 0\\
\hline
5& 4 & 4 & 0 & 0 &0 \\
\hline
6& 6 & 6 & 0 & 4 &0 \\
\hline
7& 16 & 8 & 8 & 0 &0 \\
\hline
8& 24 & 16 & 16 & 0 & 8 \\
\hline
\end{tabular}
\caption{The number of partitions with semiperimeter $n\le 8$ and degree of symmetry $k$ (see Theorem~\ref{thm:Psp}).}
\label{tab:partitions}
\end{table}

Interestingly, the distribution of the degree of symmetry on partitions with a fixed semiperimeter no longer follows a Rayleigh limit law, but rather a half-normal limit law. 

\begin{proposition}
The sequence of random variables $X_n$ giving the degree of symmetry $\ds$ of a uniformly random partition with semiperimeter $n$ has a half-normal limiting distribution. Specifically, 
\begin{equation}\label{eq:XnHn} \frac{X_n}{\sqrt{n}}\stackrel{d}{\to}\Hn\left(\frac{1}{2}\right), \end{equation}
which has density function $\frac{2\sqrt{2}}{\sqrt{\pi}}e^{-2x^2}$ for $x\ge0$, and
$$\E X_n=\sqrt{\frac{n}{2\pi}}+O(1),\qquad \V X_n=\left(\frac{1}{4}-\frac{1}{2\pi}\right)n+O(\sqrt{n}),$$
$$\Pr(X_n=k)=\frac{2\sqrt{2}}{\sqrt{\pi n}}e^{-2k^2/n}+O((k+1)n^{-3/2})+O(n^{-1})$$
uniformly for all $0\le k\le k_0\sqrt{n\log n}$ with $k_0<\frac{1}{2}$.
\end{proposition}

\begin{proof}
We apply Theorem~\ref{thm:half-normal} to the generating function from Theorem~\ref{thm:Psp} (without the constant term, for simplicity), with $\rho=1/2$, 
$$g(s,z)=\frac{(1-2z)(1+2z+2(1-s)^2z^2)}{z^2(2(s^2-2s+2)z-s^2+2s)}, \quad\text{and}\quad h(s,z)=\frac{(1-s)(1-2z+2z^2)\sqrt{1+2z}}{z^2(2(s^2-2s+2)z-s^2+2s)},$$
which satisfy $g(1,1/2)=h(1,1/2)=g_s(1,1/2)=g_{ss}(1,1/2)=0$. 
A technical detail here is that condition~(vi) does not hold, since $h(s,z)$ cannot be analytically continued to a region containing a neighborhood of $z=-1/2$. However, by singularity analysis, the first-order asymptotics is determined by the behavior near $z=1/2$, and so the proof of Theorem~\ref{thm:half-normal} given in~\cite{Wallner} still applies to this case.
The resulting limit law is half-normal with parameter
$$\sigma=\frac{\sqrt{2}\,|h_s(1,1/2)|}{1/2\,|g_z(1,1/2)|}=\frac{1}{2},$$
from where the remaining statements follow.
\end{proof}

\subsection{Symmetry by self-conjugate hooks}

Next we consider a third measure of symmetry for partitions. Following~\cite{BBES}, the boxes in the Young diagram of $\lambda\in\P$ can be decomposed into {\em diagonal hooks} as follows: the first hook is the largest hook, consisting of the first row and the first column; the second hook is the largest hook after the first hook has been removed, and so on (see Figure~\ref{fig:hooks} for an example).  The number of hooks in this decomposition equals the largest $\delta$ such that $\lambda_\delta\ge \delta$, also known as the side length of Durfee square of $\lambda$.  We define $\dsh(\lambda)$ to be the number of diagonal hooks in the Young diagram of $\lambda$ that are self-conjugate, that is, they have the same number of boxes in the row as in the column.

\begin{figure}[htb]
\centering
    \begin{tikzpicture}[scale=0.5]
      \fill[yellow] (0,-5) rectangle (1,0);
      \fill[yellow] (1,-1) rectangle (4,0);
      \fill[blue!50] (1,-4) rectangle (2,-1);
      \fill[blue!50] (2,-2) rectangle (4,-1);
      \fill[green!50] (2,-3) rectangle (3,-2);
      \draw (0,-5) grid (1,0);
      \draw (1,-1) grid (4,0);
      \draw (1,-4) grid (2,-1);
      \draw (2,-2) grid (4,-1);
      \draw (2,-3) grid (3,-2);
   \end{tikzpicture}
  \caption{The Young diagram of the partition $\lambda=(4,4,3,2,1)$ has two symmetric diagonal hooks, and so $\dsh(\lambda)=2$. Note that $\ds(\lambda)=3$ in this case, since $\lambda'=(5,4,3,2)$.}\label{fig:hooks}
\end{figure}
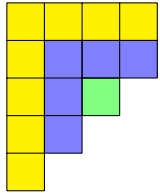

\begin{proposition}\label{prop:dsh}
The generating functions for partitions with respect to the statistic $\dsh$ and the side length of any (in the first formula) or the smallest (in the second formula) square containing their Young diagram are 
\begin{align}\label{eq:dsh}\sum_{n\ge0} \sum_{\lambda\in\P^\Box_n} s^{\dsh(\lambda)} z^n & = \frac{1}{(1-s)z + \sqrt{1-4z}}, \\
\nonumber \sum_{\lambda\in\P} s^{\dsh(\lambda)} z^{\max\{\lambda_1,\lambda'_1\}} & = \frac{1-z}{(1-s)z + \sqrt{1-4z}}.
\end{align}
\end{proposition}

\begin{proof}
Let $\lambda\in\P^\Box_n$ and $P=\partial_n(\lambda)$, where $\partial_n:\P^\Box_n\to\G_n$ is the bijection from Definition~\ref{def:partial}.
Then $\dsh(\lambda)$ equals the number of $D$ steps in the first half of $P$ that are a mirror images (with respect to reflection along $x=n$) of $U$ steps in the second half; equivalently, the number of $D$ steps of $P_L$ and $P'_R$ that coincide as segments.
Via the bijection $\bij:\G_n\to\GM^2_n$ from Definition~\ref{def:bij}, such steps become $H_1$ steps 
at height $0$ of the bicolored grand Motzkin path $\bij(P)$. 

Thus, composing the two bijections, we have that $\dsh(\lambda)=h_1^0(\bij(\partial_n(\lambda)))$, and using Lemma~\ref{lem:grand2Motzkin},
$$\sum_{n\ge0} \sum_{\lambda\in\P^\Box_n} s^{\dsh(\lambda)} z^n=\sum_{M\in\GM^2} s^{h_1^0(M)} z^{|M|}=G(z,z,s,1)=\frac{1}{(1-s)z+\sqrt{1-4z}}.$$

Finally, using the same argument as in the proof of Corollary~\ref{cor:dsmax} and noting that $\dsh(\lambda)$ does not depend on the square where the Young diagram of $\lambda$ is placed, we obtain the second formula.
\end{proof}

Applying Theorem~\ref{thm:Rayleigh} to the generating functions in Proposition~\ref{prop:dsh}, we obtain the limiting distribution of the statistic $\dsh$ on partitions. The proof of the next result is identical to that of Proposition~\ref{prop:limit-ds}.

\begin{proposition}
The sequence of random variables $X_n$ giving the value of $\dsh$ of a uniformly random partition in $\P^\Box_n$ (respectively, of a uniformly random partition $\lambda$ with $\max\{\lambda_1,\lambda'_1\}=n$) has a Rayleigh limiting distribution. Specifically, 
$$\frac{X_n}{\sqrt{n}}\stackrel{d}{\to}\Ral\left(\frac{1}{2\sqrt{2}}\right),$$ 
which has density function $8xe^{-4x^2}$ for $x\ge0$, and
$$\E X_n=\frac{\sqrt{\pi}}{4}\sqrt{n}+O(1),\qquad \V X_n=\left(\frac{1}{4}-\frac{\pi}{16}\right)n+O(\sqrt{n}),$$
$$\Pr(X_n=k)=\frac{8k}{n}e^{-4k^2/n}+O((k+1)n^{-3/2})+O(n^{-1})$$
uniformly for all $k\ge0$.
\end{proposition}

For $P\in\G_n$, denote by $\pho(P)$ the number of peaks of $P$ at height $1$, that is, occurrences of $UD$ whose middle vertex  has $y$-coordinate equal to 1.

\begin{corollary}\label{cor:dshpho}
For $n,k\ge0$,
$$|\{\lambda\in\P^\Box_n:\dsh(\lambda)=k\}|=|\{P\in\G_n:\pho(P)=k\}|.$$
\end{corollary}

\begin{proof}
We give two proofs of this equality. The first one consists of showing that 
\begin{equation}\label{eq:pho}\sum_{n\ge0}\sum_{P\in\G_n}s^{\pho(P)}z^n=\frac{1}{(1-s)z + \sqrt{1-4z}},
\end{equation}
from where the result follows by Equation~\eqref{eq:dsh}. 
Making the substitution $s=t+1$ in the left-hand side of Equation~\eqref{eq:pho}, we obtain
\begin{equation}\label{eq:pho_marked}\sum_{n\ge0}\sum_{P\in\G_n}(t+1)^{\pho(P)}z^n=\sum_{n\ge0}\sum_{P\in\G_n}\sum_{T} t^{|T|} z^n,
\end{equation}
where $T$ ranges over all subsets of the set of peaks of height one in $P$. One can think of $t$ as keeping track of {\em marked} peaks, which are an arbitrary subset of all peaks at height~$1$. 
Recall that the generating function for grand Dyck paths by semilength is simply $\frac{1}{\sqrt{1-4z}}$.
Since every grand Dyck path with marked peaks can be decomposed as a sequence of blocks consisting of a grand Dyck path followed by a marked peak (at height $1$), plus a grand Dyck path at the end, the generating function~\eqref{eq:pho_marked} equals
$$\frac{1}{1-\dfrac{tz}{\sqrt{1-4z}}}\,\frac{1}{\sqrt{1-4z}}=\frac{1}{\sqrt{1-4z}-tz}.$$
Setting $t=s-1$ we obtain the right-hand side of Equation~\eqref{eq:pho}.

The second proof is combinatorial. Consider the bijection $\psi:\P^\Box_n\to\G_n$ from \cite[Lem.\ 3.5]{BBES}, which can be defined as follows. Given
$\lambda\in\P^\Box_n$, let $\delta$ be the number of hooks in its diagonal hook decomposition described above. 
For $1\le i\le \delta$, let $a_i$ (resp.\ $\ell_i$) denote the arm length (resp.\ leg length) of the $i$th diagonal hook, defined as the number of boxes directly to the right (resp.\ below) of the corner box of the hook. Let
$$\psi(\lambda)=D^{a_\delta}U^{\ell_\delta+1}D^{a_{\delta-1}-a_\delta}U^{\ell_{\delta-1}-\ell_\delta}
\dots D^{a_{1}-a_2}U^{\ell_{1}-\ell_2}D^{n-a_1}U^{n-1-\ell_1}.$$
Then $\dsh(\lambda)=\pho(\psi(\lambda))$. Indeed, the peaks of $\psi(\lambda)$ occur at heights $1+\ell_\delta-a_\delta, 1+\ell_{\delta-1}-a_{\delta-1},
\dots, 1+\ell_1-a_1$, and so $\pho(\psi(\lambda))=|\{i:a_i=\ell_i\}|=\dsh(\lambda)$. Figure~\ref{fig:psi} shows an example of this construction.
\end{proof}

\begin{figure}[htb]
\centering
    \begin{tikzpicture}[scale=0.5]
      \draw (0,-5) rectangle (5,0);
      \fill[yellow] (0,-5) rectangle (1,0);
      \fill[yellow] (1,-1) rectangle (4,0);
      \fill[blue!50] (1,-4) rectangle (2,-1);
      \fill[blue!50] (2,-2) rectangle (4,-1);
      \fill[green!50] (2,-3) rectangle (3,-2);
      \draw (0,-5) grid (1,0);
      \draw (1,-1) grid (4,0);
      \draw (1,-4) grid (2,-1);
      \draw (2,-2) grid (4,-1);
      \draw (2,-3) grid (3,-2);
      \draw[dashed,thin] (0,0)--(3,-3);
   \end{tikzpicture}\quad
 \begin{tikzpicture}[scale=0.5] 
 	\draw (-2,0) node[above] {$\psi$};
 	\draw (-2,0) node {$\rightarrow$};
     \draw[dotted](-1,0)--(11,0);
     \draw[dotted](0,-2)--(0,2);
     \draw[thick,blue](0,0) circle(1.2pt) \U\D\D\U\U\D\U\U\D\D;
     \foreach \x in {1,5}{ \draw[orange,thick] (\x,1) circle(.15);} 
    \end{tikzpicture}
  \caption{The bijection $\psi$ from the proof of Corollary~\ref{cor:dshpho} applied to $\lambda=(4,4,3,2,1)\in\P^\Box_5$. Here $\delta=3$, $a_1=3$, $\ell_1=4$,  
$a_2=\ell_2=2$, $a_3=\ell_3=0$. The peaks at height $1$ in $\phi(\lambda)$ are highlighted in orange.}\label{fig:psi}
\end{figure}
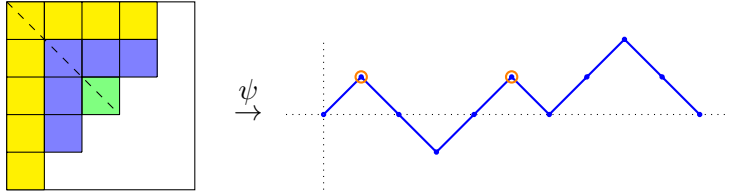

\section{Symmetry of unimodal compositions and bargraphs}\label{sec:unimodal}

A composition is a sequence of positive integers $(a_1,a_2,\dots,a_k)$ for some $k\ge1$.
We define its degree of symmetry to be the number of indices $i\le k/2$ such that $a_i=a_{k+1-i}$.
Similarly to how partitions are represented as Young diagrams, compositions can be represented as {\em bargraphs}, by arranging boxes (unit squares) into $k$ bottom-justified columns, where column $i$ from the left has $a_i$ boxes for each $i$; see Figure~\ref{fig:bargraph} for an example.
Bargraphs have been studied in the literature as a special case of column-convex polyominoes (see e.g. \cite{Fer,BMR,PB,DEbar,DEbarDyck}), and they are used in statistical physics to model polymers.

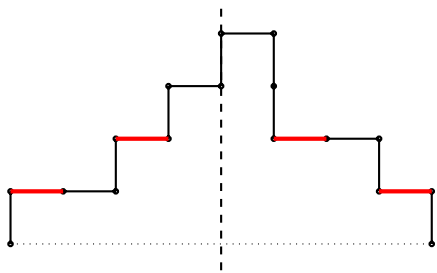
\begin{figure}[htb]
\centering
    \begin{tikzpicture}[scale=0.7]
      \draw[thin,dotted](0,0)--(8,0);
     \draw[dashed,thick](4,-.5)--(4,4.5);
     \draw[thick] (0,0) circle(1.2pt)  \N\H\H\N\H\N\H\N\H\S\S\H\H\S\H\S; 
     \draw[red,ultra thick] (0,1)--(1,1) (2,2)--(3,2) (5,2)--(6,2) (7,1)--(8,1);
\end{tikzpicture}
\caption{A unimodal bargraph $B$ with $\ds(B)=2$, $e(B)=8$ and $n(B)=4$, corresponding to the composition $(1,1,2,3,4,2,2,1)$. In this case, $B_L=EENENEN$ and $B'_R=ENEENNE$.}
\label{fig:bargraph}
\end{figure}

A bargraph can be identified with the lattice path determined by its upper boundary, namely, a self-avoiding path with steps $N=(0,1)$, $E=(1,0)$ and $S=(0,-1)$ starting at the origin and returning to the $x$-axis only at the end. 
For a bargraph $B$, let $e(B)$ denote its number of $E$ steps (also called the width of $B$), and let $n(B)$ be its number of $N$ steps. The semiperimeter of $B$ is defined as $\sp(B)=e(B)+n(B)$, 
and its degree of symmetry is defined as the degree of symmetry of the composition determined by its column heights, and denoted by $\ds(B)$.

\subsection{Unimodal bargraphs with a centered maximum}

Interpreting partitions as weakly decreasing compositions, it is natural to consider the related notion of unimodal compositions, or {\em unimodal bargraphs}; these appear, for example, in~\cite[Sec.\ 2.5]{EC1} and in~\cite[Ex.\ I.8]{Fla}, where they are called stack polyominoes. In this section we study unimodal bargraphs with a centered maximum, defined as those whose column heights satisfy $1\le a_1\le a_2\le \dots\le a_{\left\lfloor (k+1)/2 \right\rfloor}$ and 
$a_{\left\lceil (k+1)/2 \right\rceil}\ge \dots\ge a_{k-1}\ge a_k\ge1$. For example, the bargraph in Figure~\ref{fig:bargraph} is unimodal.
Let $\UB$ denote the set of unimodal bargraphs.

\begin{theorem}\label{thm:UB}
The generating function for unimodal bargraphs with a centered maximum with respect to the number of $E$ and $N$ steps and the degree of symmetry is
$$\sum_{B\in\UB}s^{\ds(B)}x^{e(B)}y^{n(B)}=\frac{y(1+x-y)}{(1-s)x^2+\sqrt{((x+1)^2-y)((x-1)^2-y)}}-y.$$
\end{theorem}

\begin{proof}
Let $B\in\UB$. If $B$ has odd width, say $2j+1$, by splitting the associated lattice path at the middle $E$ step, we can decompose it as $B=NB_LEB_RS$.
Let $B'_R$ be the path obtained by reflecting $B_R$ along the vertical line $x=j+1/2$, which passes through the center of this middle $E$ step.
Then $B_L$ and $B'_R$ are lattice paths with $N$ and $E$ steps from the origin to some common endpoint.

Suppose now that $B$ has even width, say $2j$. We can uniquely split $B$ in the middle and write $B=NB_LB_RS$ in such a way that, if 
$B'_R$ is the path obtained by reflecting $B_R$ along the vertical line $x=j$, then $B_L$ and $B'_R$ are lattice paths with $N$ and $E$ steps from the origin to some common endpoint, with the caveat that now they cannot both end with an $N$ step.

Along the lines of Definition~\ref{def:bij}, we describe a bijection between pairs of paths with $N$ and $E$ steps from the origin to a common endpoint and bicolored grand Motzkin paths.
Denoting by $\ell_i$ and $r_i$ the $i$th step of $B_L$ and $B'_R$, respectively, we construct a path $M\in\GM^2$ by letting its $i$th step be equal to 
$$\begin{cases} U & \text{ if } \ell_i=N \text{ and }r_i=E,\\
D & \text{ if } \ell_i=E \text{ and }r_i=N,\\
H_1 & \text{ if } \ell_i=r_i=E,\\
H_2 & \text{ if } \ell_i=r_i=N.
\end{cases}$$

Then $\ds(B)$ is equal to the number of common (i.e. coinciding as segments) $E$ steps of $B_L$ and $B'_R$, which in turn equals $h_1^0(M)$.
If $B$ has even width, then $$e(B)=e(B_L)+e(B'_R)=u(M)+d(M)+2h_1(M)=2(d(M)+h_1(M)).$$ 
A similar equality, but with the left-hand side replaced with $e(B)-1$, holds when $B$ has odd width. Finally, 
$$n(B)=n(B_L)+1=u(M)+h_2(M)+1.$$

With $G(x,y,s_1,s_2)$ defined as in Lemma~\ref{lem:grand2Motzkin}, it follows that the generating function for bargraphs in $\UB$ of odd width is 
$$xyG(x^2,y,s,1),$$ and the one for those of even width is
$$y[(1-y)G(x^2,y,s,1)-1].$$
In the last formula,  the $-y$ term subtracts pairs of paths $B_L$ and $B'_R$ ending both with an $N$, and the $-1$ at the end removes the possibility that both $B_L$ and $B'_R$ are empty. Summing these two expressions, we obtain
$$\sum_{B\in\UB}s^{\ds(B)}x^{e(B)}y^{n(B)}=y(1+x-y)G(x^2,y,s,1)-y,$$
which, after using Lemma~\ref{lem:grand2Motzkin}, equals the stated formula.
\end{proof}

Setting $x=z$ and $y=z$ in Theorem~\ref{thm:UB}, we obtain the generating function for unimodal bargraphs with a centered maximum where $z$ marks the semiperimeter:
\begin{equation}\label{eq:UB}
\sum_{B\in\UB}s^{\ds(B)}z^{\sp(B)}=\frac{z}{(1-s)z^2+\sqrt{1-2z-z^2-2z^3+z^4}}-z.
\end{equation}

\begin{proposition}
The sequence of random variables $X_n$ giving the degree of symmetry of a uniformly random element of $\UB$ with semiperimeter $n$ has a Rayleigh limiting distribution, that is,
$$\frac{X_n}{\sqrt{n}}\stackrel{d}{\to}\Ral\left(\frac{3-\sqrt{5}}{2\sqrt[4]{20}}\right).$$ 
\end{proposition}

\begin{proof}
We apply Theorem~\ref{thm:Rayleigh} to the bivariate generating function from Equation~\eqref{eq:UB} (removing the term $-z$ at the end for simplicity), with $\rho=\frac{3-\sqrt{5}}{2}$, $g(s,z)=(1-s)z$, and $h(s,z)=\frac{1}{z}\sqrt{(1-\rho z)(1+z+z^2)}$. The resulting Rayleigh distribution has parameter
\[
\sigma=\frac{\sqrt{2}\,|g_s(1,\rho)|}{h(1,\rho)}=\frac{\rho}{\sqrt[4]{20}}.
\qedhere
\]
\end{proof}

\subsection{A bijection with grand Motzkin paths having no peaks}\label{sec:peakless}

Next we discuss a connection between unimodal bargraphs with a centered maximum and grand Motzkin paths with no peaks. 
Consider the generating function
$$G(z^2,z,1,1)=\frac{1}{\sqrt{1-2z-z^2-2z^3+z^4}}$$
for bicolored grand Motzkin paths where steps $H_1$ and $D$ have weight $z^2$, and steps $H_2$ and $U$ have weight $z$. In this section, such paths with total weight $z^n$ will be called {\em uneven bicolored grand Motzkin paths of size $n$}; the same term without the word {\em grand} will refer to paths not allowed to go below the $x$-axis.
Setting $s=1$ in Equation~\eqref{eq:UB}, we see that a shift of $G(z^2,z,1,1)$ gives the generating function for bargraphs in $\UB$ by semiperimeter. 
The coefficients of this generating function appear as sequence A051286 in~\cite{OEIS}. 
In~\cite{BonKno}, B\'ona and Knopfmacher show that uneven bicolored grand Motzkin paths of size $n$ are in bijection with ordered pairs of compositions of $n$ into parts equal to $1$ and $2$ having the same number of parts, as well as with order ideals of size $n$ of the fence poset with $2n$ elements. Interestingly, the same generating function also enumerates grand Motzkin paths with no peaks $UD$. The goal of this subsection is to give a bijective proof of this fact.

In what follows, the terms {\em Motzkin path} and {\em grand Motzkin path}, when not preceded by the word {\em bicolored}, refer to the standard paths that only have one kind of horizontal step, denoted by $H$.

We first consider the case of paths that do not go below the $x$-axis. The generating function for uneven bicolored Motzkin paths is
\begin{equation}\label{eq:peaklessMotzkin}
F(z^2,z)=\frac{1-z-z^2-\sqrt{1-2z-z^2-2z^3+z^4}}{2z^3},
\end{equation}
which coincides with the generating function for Motzkin paths with no valleys, see~\cite[A004148]{OEIS}. Additionally, $1+zF(z^2,z)$ equals the generating function for Motzkin paths with no peaks.
Let us prove these equalities bijectively. 

\begin{proposition}\label{prop:bijections}
There are explicit bijections between the following sets:
\begin{enumerate}[(i)]
\item Motzkin paths of length $n+1$ with no peaks $UD$;
\item Motzkin paths of length $n$ with no valleys $DU$;
\item uneven bicolored Motzkin paths of size $n$.
\end{enumerate}
\end{proposition}

\begin{proof}
A simple bijection $\mu$ between (i) and (ii) can be described as follows. 
Given a Motzkin path $M$ of length $n+1$ with no peaks, construct $\mu(M)$ by replacing
every subpath of the form $DH^{i}U$ with $DH^{i+1}U$, and every subpath of the form $UH^{i+1}D$ with $UH^{i}D$, for $i\ge0$; or, if $M$ contains only $H$ steps, by simply removing an $H$. The resulting Motzkin path $\mu(M)$ has length $n$ and it contains no valleys. 

The inverse map $\mu^{-1}$ has the following description. Given a Motzkin path of length $n$ with no valleys, 
replace every subpath of the form $DH^{i+1}U$ with $DH^iU$, and every subpath of the form $UH^iD$ with $UH^{i+1}D$, for $i\ge0$; or, if $M$ contains only $H$ steps, simply add an extra $H$. 

Next we describe a bijection $\theta$ between (ii) and (iii). Let $V$ be a Motzkin path of length $n$ with no valleys. For every $D$ step of $V$, consider two possibilities: if it is followed by an $H$, simply remove this $H$ step; otherwise, find the {\em matching} $U$ step (this is the rightmost $U$ step that is at the same height and to the left of the $D$ step), change this $U$ to an $H_1$, and remove the $D$ step.
Finally, change any remaining $H$ steps to $H_2$ steps. Let $\theta(V)$ be the resulting path. Some examples are given in Figure~\ref{fig:Theta}. Since every $D$ in $V$ causes either the removal of an $H$ step or the substitution of a pair of matching $U$ and $D$ steps with an $H_1$ step, we see that $\theta(V)$ is an uneven bicolored Motzkin path of size $n$. 

The inverse map $\theta^{-1}$ has the following description. Given an uneven bicolored Motzkin path of length~$n$, first turn the $H_2$ steps into $H$ steps, and insert an $H$ after each $D$ step. Then, for each $H_1$ step, find the first $D$ to its right where the path dips below the height of this $H_1$ step, insert a $D$ immediately before this $D$ (or at the end of the path if no such $D$ was found), and turn this $H_1$ step into a $U$.
\end{proof}

\begin{figure}[htb]
\centering
   \begin{tikzpicture}[scale=0.48] 
     \draw[dotted](-1,0)--(32,0);
     \draw[dotted](0,-2)--(0,2);
     \draw[thick](0,0) circle(1.2pt) \U\H\U\H\H\D\U\H\D\D\U\H\D\H\H\D\U\H\D\D\H\U\U\H\D\U\U\U\H\D\D;

     \draw[very thick,purple] (1,1) circle(1.2pt) \H\U\H\H\D\U\H\D;
     \draw[purple] (1,.5) rectangle (9,2.5);
     \draw[purple] (1.5,2) node {$M_1$};
	 \draw[purple,->] (5,.5)-- node[right]{$\mu$} (6,-1);     
     \draw[very thick,purple] (3,-2.5) circle(1.2pt) \H\U\H\D\H\U\D;
     \draw[purple] (3,-3) rectangle (10,-1);     
	 \draw[purple,->] (7,-3)-- node[right]{$\theta$} (8,-4.5);     

     \draw[very thick,orange] (11,1) circle(1.2pt) \H;
     \draw[orange] (11,.5) rectangle (12,1.5);
     \draw[orange] (11.5,2) node {$M_2$};
   	 \draw[orange,->] (11.5,.5)-- node[right]{$\mu$} (12,-2.2);  
     \draw[very thick,orange] (12,-2.5) circle(1.2pt);
     \draw[orange] (11.7,-2.8) rectangle (12.3,-2.2);   
   	 \draw[orange,->] (12.1,-2.8)-- node[right]{$\theta$} (12.9,-5.7);   
   	 
	\draw[very thick, olive] (13,0) circle(1.2pt) \H\H\D\U\H\D\D\H\U\U\H\D\U;    	 
       \draw[olive] (13,-2.5) rectangle (26,.5); 	 
     \draw[olive] (20.5,-1) node {$\overline{V_2}$};
   	 \draw[olive,->] (19.5,-2.5)-- node[right]{$\overline{\theta}$} (19,-6.5);

	\draw[very thick, cyan] (27,1) circle(1.2pt) \U\H\D;    	 
       \draw[cyan] (27,.5) rectangle (30,2.5); 	 
     \draw[cyan] (28.5,1) node {$M_3$};
     	 \draw[cyan,->] (28.5,.5)-- node[right]{$\mu$} (28.2,-1);     
     \draw[very thick,cyan] (27,-2.5) circle(1.2pt) \U\D;
     \draw[cyan] (27,-3) rectangle (29,-1);     
	 \draw[cyan,->] (27.5,-3)-- node[right]{$\theta$} (24.7,-5.5);  
	 
	 \draw[->] (.5,-1) to [out=270,in=160] node[left]{$\Theta$} (4   ,-6);

\begin{scope}[shift={(5,-7)}]    
     \draw[dotted](-1,0)--(22,0);
     \draw[dotted](0,-2)--(0,2); 
     \draw[thick](0,0) circle(1.2pt) \U coordinate(last); \drawHtwo[purple]
     \draw[very thick,purple](last) circle (1.2pt) \U coordinate(last); \drawHtwo[purple]
     \draw[very thick,purple](last) circle (1.2pt) \D\Hone coordinate(last);
     \draw[purple] (1,.5) rectangle (6,2.5); 
     \draw[thick](last) \D\U\D coordinate(last); 
     \draw[very thick,orange] (8,1) circle(1.2pt);
     \draw[orange] (7.7,.7) rectangle (8.3,1.3);   
     \drawHtwo[olive]\drawHtwo[olive]
     \draw[very thick,olive](last) circle (1.2pt) \D\U\D\Hone coordinate(last); 
     \drawHtwo[olive]
     \draw[very thick,olive](last) circle (1.2pt) \U\Hone coordinate(last);
     \draw[olive] (9,-1.5) rectangle (18,.5);   
          \draw[thick](last) circle(1.2pt) \U coordinate(last);
          \draw[very thick,cyan](last)   \Hone coordinate(last);
     \draw[cyan] (19,.5) rectangle (20,1.5);   
     \draw[thick](last) circle(1.2pt) \D;
\end{scope}
    \end{tikzpicture}
    
   \caption{The bijection $\Theta$ from Theorem~\ref{thm:Theta}, constructed by applying the bijections $\mu$ and $\theta$ from Proposition~\ref{prop:bijections} to the appropriate subpaths. As before, $H_1$ and $H_2$ steps are pictured with double lines and wavy lines, respectively, and the overline indicates a reflection along the $x$-axis.}
   \label{fig:Theta}
\end{figure}
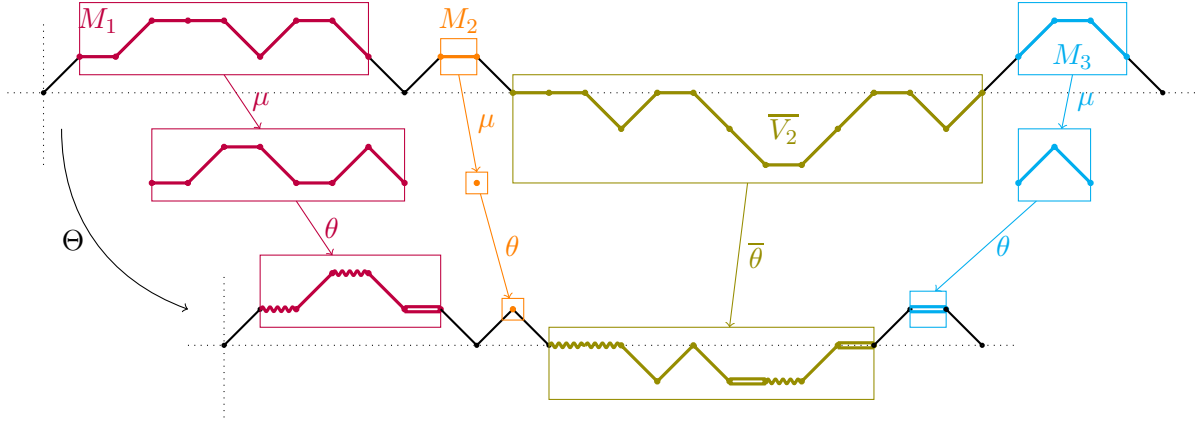

Combining the bijections from Proposition~\ref{prop:bijections}, we can now prove the following result.

\begin{theorem}\label{thm:Theta}
There is an explicit bijection $\Theta$ between the set of grand Motzkin paths of length $n$ with no peaks $UD$, and the set of uneven bicolored grand Motzkin paths of size $n$.
\end{theorem}

\begin{proof}
By splitting the blocks that stay strictly above the $x$-axis, every grand Motzkin path $M$ of length~$n$ with no peaks can be decomposed uniquely as 
$$
M=\overline{V_0}\;UM_1D\;\overline{V_1}\;UM_2D\;\overline{V_2}\;\dots\;UM_rD\;\overline{V_r}
$$ 
for some $r\ge0$, where each $M_i$ is a non-empty Motzkin path with no peaks, and each $\overline{V_i}$ is the reflection along the $x$-axis of a (possibly empty) Motzkin path $V_i$ with no valleys.
Applying now the bijections $\mu$ and $\theta$ from Proposition~\ref{prop:bijections}, $\mu(M_i)$ is a Motzkin path of length $|M_i|-1$ with no valleys for all $i$, 
and $\theta(\mu(M_i))$ and $\theta(V_i)$ are uneven bicolored Motzkin paths of size $|M_i|-1$ and $|V_i|$, respectively. 

We define the image of $M$ to be the uneven bicolored grand Motzkin path
$$\Theta(M)=\overline{\theta(V_0)}\;U\theta(\mu(M_1))D\;\overline{\theta(V_1)}\;U\theta(\mu(M_2))D\;\overline{\theta(V_2)}\;\dots\;U\theta(\mu(M_r))D\;\overline{\theta(V_r)},$$
where $\overline{\theta(V_i)}$ denotes the reflection of $\theta(V_i)$ along the $x$-axis; see Figure~\ref{fig:Theta} for an example. Since the size of each $\overline{\theta(V_i)}$ is $|V_i|$, and the size of each block $U\theta(\mu(M_i))D$ is
$1+(|M_i|-1)+2=|UM_iD|$, the total size of $\Theta(M)$ is $|M|=n$. 
Since $\mu$ and $\theta$ are bijections by Proposition~\ref{prop:bijections}, it is clear that $\Theta$ is a bijection as well.
\end{proof}

\section{Symmetry of Dyck paths}\label{sec:Dyck}

Our goal in this section is to study the generating function
$$D(s,z)=\sum_{n\ge0} \sum_{P\in\DP_{n}} s^{\ds(P)} z^n$$
for Dyck paths with respect to their degree of symmetry. 
In contrast to the simplicity of the generating function in Theorem~\ref{thm:sym_G} for grand Dyck paths, the generating function $D(s,z)$ is unwieldy. To investigate the statistic $\ds$ on Dyck paths, we will first rephrase the problem in terms of walks in the plane. Then we will apply some transformations on the walks that will allow us to obtain a functional equation for a refinement of $D(s,z)$.

Let $\WQ^1_n$ denote the set of walks constrained to the first quadrant $\{(x,y):x,y\ge0\}$, starting at the origin, ending on the diagonal $y=x$, and having $n$ steps from the set $\{\NE,\NW,\SE,\SW\}$, where we use the notation $\NE=(1,1)$, $\NW=(-1,1)$, $\SE=(1,-1)$, $\SW=(-1,-1)$.

We start by describing a bijection $\bijw:\DP_n\to\WQ^1_n$, which, in a similar form, has been used in~\cite{Eliwalks,G-B86,BMM}. 
Given $P\in\DP_n$, first define two paths as in Section~\ref{sec:GD}: $P_L$ is the left half of $P$, and $P'_R$ is the path obtained by reflecting the right half of $P$ along the vertical line $x=n$. Both $P_L$ and $P'_R$ are paths  with steps $U$ and $D$ from $(0,0)$ to the same point on the line $x=n$, not going below the $x$-axis.
Denote the directions of the $i$th steps of $P_L$ and $P'_R$ by $\ell_i,r_i\in\{U,D\}$, respectively.
Let $\bijw(P)\in\WQ^1_n$ be the walk whose $i$th step is equal to 
$$\begin{cases} 
\NE & \text{ if } \ell_i=r_i=U,\\
\NW & \text{ if } \ell_i=U \text{ and }r_i=D,\\
\SE & \text{ if } \ell_i=D \text{ and }r_i=U,\\
\SW & \text{ if } \ell_i=r_i=D.\\
\end{cases}$$

Under this bijection, symmetric steps of $P$, which correspond to common steps of $P_L$ and $P'_R$, become steps of $\bijw(P)$ lying entirely on the diagonal $y=x$.
Finding an equation for the generating function of walks with a variable marking the number of such steps would be troublesome, since in order to contain a term corresponding to walks ending on the diagonal $y=x$, it would require taking diagonals of generating functions (as defined in Equation~\eqref{eq:diag-def}).

To circumvent this problem, we modify the walks so that the steps that we need to keep track of lie on the boundary of the region.
Folding walks in $\WQ^1_n$ along the diagonal $y=x$ (that is, reflecting the steps above the diagonal onto steps below it), we obtain walks in the first octant $\{(x,y):x\ge y\ge0\}$. In order not to lose information while folding, we have to allow the resulting walks in the octant to use two colors for steps $\SE$ leaving the diagonal. These colors keep track of whether the portion of the walk between the colored step and the next return to the diagonal was above or below the diagonal on the original quadrant walk. We obtain a bijection between $\WQ^1_n$ and the set $\WQ^2_n$ of walks in the first octant starting at the origin, ending on the diagonal $y=x$, having $n$ steps from the set $\{\NE,\NW,\SE,\SW\}$, and where steps $\SE$ leaving the diagonal $y=x$ have two possible colors.

Next we apply the linear transformation $(x,y)\mapsto(y,\frac{x-y}{2})$ from the first octant to the first quadrant. This transformation gives a bijection between $\WQ^2_n$ and the set $\WQ^3_n$ of walks in the first quadrant starting at the origin, ending on the $x$-axis, having $n$ steps from the set $\{E,W,\NW,\SE\}$ (denoting $E=(1,0)$ and $W=(-1,0)$), and where steps $\NW$ leaving the $x$-axis have two possible colors. Under this bijection, steps of walks in $\WQ^2_n$ lying on the diagonal $y=x$ become steps of walks in $\WQ^3_n$ lying on the $x$-axis. Table~\ref{tab:WQ} summarizes the above sequence of bijections from $\DP_n$ to $\WQ^3_n$. 
It follows that $D(s,z)$ is the generating function for walks in $\WQ^3_n$ where $z$ marks the length (total number of steps) and $s$ marks the number of steps on the $x$-axis.

\begin{table}[htb]
\centering
\hspace{24mm}
$\DP_n \hspace{6mm} \stackrel{\bijw}{\longrightarrow} \hspace{6mm} \WQ^1_n \hspace{4mm} \stackrel{\text{fold along y=x}}{\longrightarrow} \hspace{4mm} \WQ^2_n \hspace{3.5mm}\stackrel{(x,y)\mapsto(y,\frac{x-y}{2})}\longrightarrow \hspace{3.5mm} \WQ^3_n$
\begin{tabular}{r||c|c|c|c}
\hline\hline
walks in & 
\begin{tikzpicture}[scale=1]
\draw[fill=yellow] (0.7071,0.7071)--(0,0)--(1,0);
\draw (1.5,.5) node {first};
\draw (1.5,.2) node {octant};
\end{tikzpicture}
& 
\begin{tikzpicture}[scale=.9]
\draw[fill=yellow] (0,1)--(0,0)--(1,0);
\draw (1.5,.6) node {first};
\draw (1.5,.3) node {quadrant};
\end{tikzpicture}
 & \begin{tikzpicture}[scale=1]
\draw[fill=yellow] (0.7071,0.7071)--(0,0)--(1,0);
\draw (1.5,.5) node {first};
\draw (1.5,.2) node {octant};
\end{tikzpicture}
& 
\begin{tikzpicture}[scale=.9]
\draw[fill=yellow] (0,1)--(0,0)--(1,0);
\draw (1.5,.6) node {first};
\draw (1.5,.3) node {quadrant};
\end{tikzpicture} \\ \hline
allowed steps & 
\begin{tikzpicture}[scale=.3]
\draw[thick,->] (0,0)--(1,1) node[right] {$U=\NE$};
\draw[thick,->] (0,0)--(1,-1) node[right] {$D=\SE$};
\end{tikzpicture}
& 
\begin{tikzpicture}[scale=.3]
\draw[thick,->] (0,0)--(1,1) node[right] {$\NE$};
\draw[thick,->] (0,0)--(1,-1) node[right] {$\SE$};
\draw[thick,->] (0,0)--(-1,1) node[left] {$\NW$};
\draw[thick,->] (0,0)--(-1,-1) node[left] {$\SW$};
\end{tikzpicture}
& 
\begin{tikzpicture}[scale=.3]
\draw[thick,->] (0,0)--(1,1) node[right] {$\NE$};
\draw[thick,->] (0,0)--(1,-1) node[right] {$\SE$};
\draw[thick,->] (0,0)--(-1,1) node[left] {$\NW$};
\draw[thick,->] (0,0)--(-1,-1) node[left] {$\SW$};
\end{tikzpicture}
& 
\begin{tikzpicture}[scale=.3]
\draw[thick,->] (0,0)--(1,0) node[right] {$E$};
\draw[thick,->] (0,0)--(1,-1) node[right] {$\SE$};
\draw[thick,->] (0,0)--(-1,1) node[left] {$\NW$};
\draw[thick,->] (0,0)--(-1,0) node[left] {$W$};
\end{tikzpicture}\\ \hline
length & $2n$ & $n$ & $n$ & $n$ \\ \hline
ending on & $x$-axis & diagonal & diagonal & $x$-axis \\ \hline
2 colors for & - & - & $\SE$ leaving diagonal & $\NW$ leaving $x$-axis \\ \hline
$\ds$ counts & symmetric steps & steps on diagonal & steps on diagonal & steps on $x$-axis
\end{tabular}
\caption{A summary of the bijections between Dyck paths and walks used to derive Theorem~\ref{thm:symDyck}.}
\label{tab:WQ}
\end{table}

To enumerate walks in $\WQ^3_n$, we consider a more general class of walks where any endpoint is allowed.
Let $R(x,y,s,z)$ be the generating function where the coefficient of $x^i y^j s^k z^n$ is the number of walks in the first quadrant with $n$ steps from the set $\{E,W,\NW,\SE\}$, starting at the origin, ending at $(i,j)$, having $k$ steps entirely on the $x$-axis, and where steps $\NW$ leaving the $x$-axis have two possible colors. The generating function for Dyck paths with respect to the statistic $\ds$ is the specialization $D(s,z)=R(1,0,s,z)$.
By considering the different possibilities for the last step of the walk, we obtain the following functional equation for $R(x,y):=R(x,y,s,z)$.
\begin{align*}
R(x,y)=& \ 1+z\left(x+\frac{1}{x}+\frac{x}{y}+\frac{y}{x}\right)R(x,y) &\\
& - z \left(\frac{1}{x}+\frac{y}{x}\right)R(0,y)  & \text{(no steps $W$ or $\NW$ when on $y$-axis)}\\
& - z\, \frac{x}{y} R(x,0)  & \text{(no steps SE when on $x$-axis)}\\
& + z\, \frac{y}{x} \left(R(x,0)-R(0,0)\right)  & \text{(second color for steps $\NW$ when on positive $x$-axis)}\\
& + z\, (s-1) \left(x+\frac{1}{x}\right)R(x,0)  & \text{(steps $E$ and $W$ on $x$-axis have weight $s$)}\\
& - z\, (s-1) \frac{1}{x} R(0,0).  & \text{(but no step $W$ when at the origin)}
\end{align*}

Collecting terms in this equation, we can summarize our result as follows.

\begin{theorem}\label{thm:symDyck}
The generating function for Dyck paths with respect to their degree of symmetry is $D(s,z)=R(1,0,s,z)$, where $R(x,y):=R(x,y,s,z)$ satisfies the functional equation
\begin{multline*} 
\left(xy-z(y+x^2)(1+y)\right)R(x,y)\\
=xy-zy(1+y)R(0,y)+z\left(y^2-x^2+(s-1)y(x^2+1)\right)R(x,0)-zy(y+s-1)R(0,0).
\end{multline*}
\end{theorem}

The first few coefficients of $D(s,z)$ are given in Table~\ref{tab:Dyck}.
Even though we have been unable to solve this functional equation, it suggests that the generating function $D(s,z)=R(1,0,s,z)$ is D-finite. Computations by Alin Bostan (personal communication, July 2019) using Theorem~\ref{thm:symDyck}  have led to the following conjecture.

\begin{conjecture}
The generating function $D(s,z)$ is D-finite in $z$ but not algebraic. Specifically, it satisfies
a fifth order linear differential equation with polynomial coefficients with maximum degree 27 in $z$.
\end{conjecture}

Removing the choice of two colors for steps $\NW$ leaving the $x$-axis, and setting the weight $s=1$ for steps on the $x$-axis, the resulting walks in the first quadrant with steps in $\{E,W,\NW,\SE\}$ were studied by Gouyou-Beauchamps~\cite{G-B86}, and they are in bijection with walks in the first octant with steps in $\{N,S,E,W\}$. Those ending on the $x$-axis correspond to Dyck paths whose left half $P_L$ is weakly above the folded right half $P'_R$, counted by sequence A005817 in~\cite{OEIS}, and those with no restriction on the endpoint are counted by sequence A005558 in~\cite{OEIS}. 
While the functional equation for Gouyou-Beauchamps walks can be solved using the kernel method (see \cite[Sec.\ 5.3]{BMM}), we have not been able to adapt this method to solve the equation in Theorem~\ref{thm:symDyck}. Some possible approaches and related work are discussed in Section~\ref{sec:open}.

Because of the above bijection between $\DP_n$ and $\WQ_n^3$, several specializations of $R(x,y,1,z)$ give rise to known sequences. For example, 
the coefficient of $z^{2n}$ in $R(0, 0, 1, z)$ is $C_n^2$, which is sequence A001246 in \cite{OEIS}, and more generally, $R(x,0,1,z)$ is the generating function for Dyck paths where $x$ marks the height of their midpoint \cite[A213600]{OEIS}; this statistic will be discussed in Section~\ref{sec:other}. One can extend the bijection by allowing arbitrary endpoints for the walks, which corresponds to allowing the halves $P_R$ and $P'_L$ to end at different heights (equivalently, we can think of $P$ as a Dyck path that is allowed to be discontinuous at its midpoint). It follows that the coefficient of $z^{n}$ in $R(1, 1, 1, z)$ is $\binom{n}{\lfloor n/2 \rfloor}^2$, the number of pairs $(P_R,P'_L)$ with endpoints at arbitrary heights; this is sequence A018224 in~\cite{OEIS}. Similarly, changing the restriction on $\WQ_n^3$ so that walks end on the $y$-axis translates into changing the restriction on $\WQ_n^1$ so that walks end on either axis, which in turn translates into one of the halves $P_R$ and $P'_L$ ending at height $0$, while there are no restrictions on the other. It follows that the coefficient of $z^{2n}$ in $R(0, 1, 1, z)$ is $2C_n\binom{2n}{n}-C_n^2=\frac{(2n)!(2n+1)!}{n!^2(n+1)!^2}$, which is sequence A000891 in~\cite{OEIS}.

\begin{table}[h]
\centering
\begin{tabular}{c@{\qquad}c}
$\card{\{P \in \DP_n: \ds(P) = k\}}$ & $\card{\{P \in \DP_n: \sv(P) = k\}}$ \smallskip \\
\begin{tabular}{c||r|r|r|r|r|r|r|r|} 
 $n \setminus k$ & 1 & 2 & 3 & 4 & 5 & 6 & 7\\ 
\hline\hline
1& 1& & & & & &   \\
\hline
2&0 &2 & & & & & \\
\hline
3& 2 &0 &3 & & & & \\
\hline
4&2 &6 &0 & 6 &  & & \\
\hline
5& 8 & 8 & 16 & 0 & 10 &  & \\
\hline
6& 16 & 32 & 24 & 40 & 0 & 20 &  \\
\hline
7& 52 & 84 & 108 & 60 & 90 & 0 & 35 \\
\hline
\end{tabular}
&
\begin{tabular}{c||r|r|r|r|r|r|r|r|} 
 $n \setminus k$ & 1 & 2 & 3 & 4 & 5 & 6 & 7\\ 
\hline\hline
1& 1& & & & & &   \\
\hline
2&0 &2 & & & & & \\
\hline
3&0 &2 &3 &  & & & \\
\hline
4&0 &2 &6 & 6 &  & & \\
\hline
5& 0 & 4 & 12 & 16 & 10 &  & \\
\hline
6& 0 & 8 & 24 & 40 & 40 & 20 &  \\
\hline
7& 0 & 20 & 60 & 104 & 120 & 90 & 35 \\
\hline
\end{tabular}
\end{tabular}
\caption{The number of Dyck paths of length $n\le 7$ with a given degree of symmetry (left, see Theorem~\ref{thm:symDyck}) and with a given number of symmetric vertices (right, see Theorem~\ref{thm:svDyck}).}
\label{tab:Dyck}
\end{table}

In analogy with Theorem~\ref{thm:sv_G} for grand Dyck paths, next we consider an alternative measure of the symmetry of a Dyck path, given by its number of {\em symmetric vertices}. For example, the Dyck path at the bottom of Figure~\ref{fig:symD} has 5 symmetric vertices. 

Denote the generating function for Dyck paths with respect to the number of symmetric vertices by
$$\vD(v,z)=\sum_{n\ge0}\sum_{P\in\DP_n}v^{\sv(P)}z^n.$$ 
Through the sequence of bijections in Table~\ref{tab:WQ}, symmetric vertices of Dyck paths become vertices on the $x$-axis of walks in $\WQ^3_n$, not counting the endpoint of the walk. The argument that we used to obtain an equation for $R(x,y,s,z)$ can be modified to obtain the following equation for the generating function $\vR(x,y):=\vR(x,y,v,z)$, where the coefficient of $x^i y^j v^k z^n$ is the number of walks in the first quadrant with $n$ steps from the set $\{E,W,\NW,\SE\}$, starting at the origin, ending at $(i,j)$, having $k$ steps starting on the $x$-axis, and where steps $\NW$ leaving the $x$-axis can have 2 colors. 
\begin{align*}
\vR(x,y)=& \ 1+z\left(x+\frac{1}{x}+\frac{x}{y}+\frac{y}{x}\right)\vR(x,y) &\\
& - z \left(\frac{1}{x}+\frac{y}{x}\right)\vR(0,y)  & \text{(no steps $W$ or $\NW$ when on $y$-axis)}\\
& - z\, \frac{x}{y} \vR(x,0)  & \text{(no steps SE when on $x$-axis)}\\
& + z\, \frac{y}{x} \left(\vR(x,0)-\vR(0,0)\right)  & \text{(second color for steps $\NW$ when on positive $x$-axis)}\\
& + z\, (v-1) \left(x+\frac{1}{x}+\frac{2y}{x}\right)\vR(x,0)  & \text{(steps $E, W, \NW$ leaving $x$-axis have weight $v$)}\\
& - z\, (v-1) \left(\frac{1}{x}+\frac{2y}{x}\right) \vR(0,0).  & \text{(but no steps $W, \NW$ when at the origin)}
\end{align*}

\begin{theorem}\label{thm:svDyck}
The generating function for Dyck paths with respect to their number of symmetric vertices is $\vD(v,z)=\vR(1,0,v,z)$, where $\vR(x,y):=\vR(x,y,v,z)$ satisfies the functional equation
\begin{multline*}
\left(xy-z(y+x^2)(1+y)\right)\vR(x,y)\\
=xy-zy(1+y)\vR(0,y)+z\left(y^2-x^2+(v-1)y(x^2+1+2y)\right)\vR(x,0)-zy\left(y+(v-1)(1+2y)\right)\vR(0,0).
\end{multline*}
\end{theorem}

The first few coefficients of the generating functions $D(s,z)$ and $\vD(v,z)$ appear in Table~\ref{tab:Dyck}.

\section{The height of the midpoint in Dyck paths}\label{sec:other}

Given the simple algebraic generating function in Theorem~\ref{thm:sym_G} for the number of grand Dyck paths with respect to their degree of symmetry, one may not have expected that
the enumeration of Dyck paths with respect to the same statistic would be significantly more complicated.

In this section we shed some light on this phenomenon, by considering a statistic that is simpler than the degree of symmetry, namely the height of the midpoint of the path. Even though we can enumerate both Dyck paths and grand Dyck paths with respect to this statistic, the generating function for grand Dyck paths is again algebraic, while the one for Dyck paths is not. The height of the midpoint in Dyck paths and grand Dyck paths was studied by Janse van Rensburg in~\cite{JvR}, where these paths serve as models of linear polymers subject to an external force.

We start by defining generalizations of Dyck paths that start and end at arbitrary heights.
For $a,b\ge0$, let $\B_n^{(a,b)}$ denote the set of lattice paths with steps $U$ and $D$ from $(0,a)$ to $(n,b)$ that do not go below the $x$-axis. Let $\B^{(a,b)}=\bigcup_{n\ge0}\B_n^{(a,b)}$.
By definition, $\B_{2n}^{(0,0)}=\DP_n$.
In the formulas in this subsection, we will assume that $n\equiv a+b\bmod 2$, since otherwise $\B_n^{(a,b)}=\emptyset$. 
Define the generating functions 
$$
B^{(a,b)}(z)=\sum_{n\ge0} |\B_n^{(a,b)}| \,z^n \quad \text{and} \quad B(u,v,z)=\sum_{a,b\ge0}B^{(a,b)}(z)\,u^av^b.
$$
Note that $B^{(0,0)}(z)=\Cat(z^2)$, where $\Cat(z)$ is the generating function of the Catalan numbers, defined in Equation~\eqref{eq:Cat}.

\begin{lemma}\label{lem:Bab}
\begin{enumerate}[(i)]
\item For $n,a,b\ge0$ with $n\equiv a+b\bmod 2$, $$|\B_n^{(a,b)}|=\binom{n}{\frac{n-b+a}{2}}-\binom{n}{\frac{n-b-a-2}{2}}.$$
\item For $0\le a\le b$, 
$$B^{(a,b)}(z)= z^{b-a} C(z^2)^{b-a+1}\,\frac{1-\left(zC(z^2)\right)^{2a+2}}{1-\left(zC(z^2)\right)^2}.$$
\item $$B(u,v,z)=\frac{2}{(1-uv)\left(1+uv-2(u+v)z+(1-uv)\sqrt{1-4z^2}\right)}.$$
\end{enumerate}
\end{lemma}

\begin{proof}
Part (i) is obtained using the reflection principle~\cite{Andre}, since paths from $(0,a)$ to $(n,b)$ that go below the $x$-axis are in bijection with paths from $(0,a)$ to $(n,-b-2)$.

For part (ii), note that for $a\le b$, every path $P\in\B^{(a,b)}$ can be decomposed uniquely as
$$P=A_1 D A_2 D \dots D A_{i+1} U A_{i+2} U \dots U A_{2i+1}\, U A'_{1} U A'_{2}\dots U A'_{b-a},$$
for some $0\le i\le a$, where the $A_j$ and the $A'_j$ are Dyck paths for all $j$. It follows that
$$B^{(a,b)}(z)=\left(z C(z^2)\right)^{b-a} \sum_{i=0}^a z^{2i} C(z^2)^{2i+1}= 
z^{b-a} C(z^2)^{b-a+1}\,\frac{1-\left(zC(z^2)\right)^{2a+2}}{1-\left(zC(z^2)\right)^2}.$$

For part (iii), first consider paths that start at an arbitrary height (marked by $u$), end at an arbitrary height (marked by $v$), do not go below the $x$-axis, but touch the $x$-axis.
Since each such path starting at height $a$ and ending at height $b$ can be decomposed as 
$$A_a D  \dots D A_1 D A_{0} U A'_{1} U \dots U A'_{b},$$
where the $A_j$ and the $A'_j$ are Dyck paths for all $j$,
their generating function is $$\frac{C(z^2)}{(1-uzC(z^2))(1-vzC(z^2))}= \frac{2}{1+uv-2(u+v)z+(1-uv)\sqrt{1-4z^2}}.$$
Multiplying by $\frac{1}{1-uv}$ to account for vertical translations of these paths gives the desired expression.
\end{proof}

For $P\in\G_n$, let $\hm(P)$ denote the height of the midpoint of $P$, so that this point has coordinates $(n,\hm(P))$.
To enumerate Dyck paths with respect to the statistic $\hm$, recall from Lemma~\ref{lem:Bab}(i) that 
$$|\B^{(0,b)}_n|=\binom{n}{\frac{n-b}{2}}-\binom{n}{\frac{n-b}{2}-1}=\frac{2b+2}{n+b+2}\binom{n}{\frac{n-b}{2}}$$
if $n-b$ is even and $0\le b\le n$, and $|\B^{(0,b)}_n|=0$ otherwise.
The generating function of these numbers, often called ballot numbers, is 
\begin{equation}\label{eq:Bob}\sum_{n,b\ge0} |\B^{(0,b)}_n|\, y^bz^n=B(0,y,z)=\frac{2}{1-2yz+\sqrt{1-4z^2}}=\frac{C(z^2)}{1-yzC(z^2)}
\end{equation} 
by Lemma~\ref{lem:Bab}(iii). This formula appears in~\cite{OTW} in the context of directed polymers.

Let now
$$
H(y,z)=\sum_{n\ge0}\sum_{P\in\DP_n}y^{\hm(P)}z^n
$$
be the generating function of Dyck paths with respect to the height of their midpoint. Note that $H(1,z)=C(z)$ and $H(x,z)=R(x,0,1,z)$, with $R$ defined in Section~\ref{sec:Dyck}. Then
$$H(y,z)=\sum_{n,b\ge0} |\B^{(0,b)}_n|^2\, y^{b}z^n.$$ 
In contrast to Equation~\eqref{eq:Bob}, this generating function is not algebraic, since the coefficients of $H(0,z)=\sum_{m\ge0}C_m^2 z^{2m}$ grow asymptotically as $C_m^2\sim \frac{16^m}{m^3\pi}$, which is not a possible asymptotic behavior for coefficients of an algebraic generating function (see~\cite{Jun}). Define the diagonal of $A(z_1,z_2)=\sum_{i,j}a_{i,j}z_1^iz_2^j$ to be 
\begin{equation}\label{eq:diag-def} \diag_{z_1,z_2}^z A=\sum_n a_{n,n} z^n;
\end{equation} see~\cite[Sec.\ 6.3]{EC2}. 
Using Equation~\eqref{eq:Bob}, we can express $H(y,z)$ as a diagonal of an algebraic generating function as 
$$H(y,z)=\diag_{z_1,z_2}^z \frac{C(z_1^2)C(z_2^2)}{1-yz_1z_2C(z_1^2)C(z_2^2)},$$
which implies by~\cite{Lip} that $H(y,z)$ is D-finite.
An expression for this generating function obtained using the kernel method is given in \cite{JvR}, where the Dyck path is used as a model of a linear polymer with its endpoints attached a wall, having its midpoint pulled away by an external force. 

The generating function for Dyck paths whose midpoint has a fixed height $b$ is
the Hadamard product of $B^{(0,b)}(z)$ with itself. 
By Lemma~\ref{lem:Bab}(ii), $B^{(0,b)}(z)=z^bC(z^2)^{b+1}$ is algebraic, and so its Hadamard product with itself is D-finite, even though it is not algebraic.

\medskip

For comparison, let us show that the enumeration of grand Dyck paths with respect to the height of their midpoint is straightforward and yields an algebraic generating function. The associated polymer model is also studied in~\cite{JvR}.

\begin{proposition}\label{prop:hm_G} The generating function for grand Dyck paths with respect to their semilength and the height of their midpoint is
\begin{equation}\label{eq:hm_G}
\sum_{n\ge0}\sum_{P\in\G_n}y^{\hm(P)}z^n=\frac{1}{\sqrt{(1-yz-z/y)^2-4z^2}}=\sum_{n\ge0}\sum_{k=0}^n \binom{n}{k}^2 y^{n-2k}z^n.
\end{equation}
\end{proposition}

\begin{proof}
Let $P\in\G_n$, and let $M=\bij(P)\in\GM^2$, where $\bij$ is the bijection from Definition~\ref{def:bij}. Then $\hm(P)$ equals the number of $U$ steps minus the number of $D$ steps of the left half of $P$, which in turn equals $u(M)+h_2(M)-d(M)-h_2(M)$. It follows that
the left-hand side of Equation~\eqref{eq:hm_G} is equal to
$$\sum_{M\in\GM^2}y^{u(M)+h_2(M)-d(M)-h_2(M)}z^{|M|}=G(z/y,yz,1,1)=\frac{1}{\sqrt{(1-yz-z/y)^2-4z^2}},$$
by Lemma~\ref{lem:grand2Motzkin}.

The coefficient of $y^{n-2k}z^n$ in this generating function can be also computed directly, since in order to construct a path $P\in\G_n$ with $\hm(P)=n-2k$, there are $\binom{n}{k}$ choices for the left half and  $\binom{n}{k}$ choices for the right half.
\end{proof}

Taking square roots of the coefficients in the generating function~\eqref{eq:hm_G}, which is equivalent to counting left halves of grand Dyck paths with respect to their ending height, we obtain a rational generating function
$$\sum_{n\ge0}\sum_{k=0}^n \binom{n}{k} y^{n-2k}z^n=\frac{1}{1-(y+1/y)z}.$$
The fact that the generating function for grand Dyck paths with respect to the statistic $\hm$ is the diagonal of a rational generating function explains why it is algebraic; see \cite[Thm.\ 6.3.3]{EC2}.

Finally, we remark that the method used to prove Theorem~\ref{thm:sym_G} and Proposition~\ref{prop:hm_G} also gives a common generalization of the two generating functions, namely
$$\sum_{n\ge0}\sum_{P\in\G_n}s^{\ds(P)}y^{\hm(P)}z^n=\frac{1}{(1-s)(y+1/y)z+\sqrt{(1-yz-z/y)^2-4z^2}}.$$

\section{Further research}\label{sec:open}

In this paper we have given generating functions for a few combinatorial objects with respect to their degree of symmetry. In all cases, the size of the object was measured by a ``one-dimensional'' parameter, such as the length (or semilength) of a path in Sections~\ref{sec:GD} and~\ref{sec:Dyck}, the semiperimeter of a partition or the side length of a square containing its Young diagram in Section~\ref{sec:partitions}, and the semiperimeter of a bargraph in Section~\ref{sec:unimodal}. 

For some combinatorial objects, such as partitions and compositions ---represented as Young diagrams and bargraphs, respectively---, it is also natural to consider another measure of size, namely the sum of the entries (equivalently, the area). Using this ``two-dimensional" parameter as the size function, the degree of symmetry in compositions is studied in~\cite{DEasym}. However, in many cases, 
the generating functions with respect to area and degree of symmetry are unknown.

\begin{problem}
Find the generating function for partitions and unimodal compositions with respect to area and degree of symmetry.
\end{problem}

In~\cite{BOR}, Beaton, Owczarek, and Rechnitzer consider the enumeration of quarter plane walks with respect to the number of vertices that lie on each axis, and they solve several cases, including that of Gouyou-Beauchamps walks (without our modification giving two colors to $\NW$ steps leaving the $x$-axis). Such walks, when required to end on the $x$-axis, are in bijection with Dyck paths whose left half $P_L$ is weakly above the folded right half $P'_R$, and vertices of the walk that lie on the $x$-axis become symmetric vertices of the corresponding Dyck path. This case, formulated as a pair of interacting directed walks, is also studied in depth by Tabbara, Owczarek, and Rechnitzer in~\cite{TOR}. The tools in these papers may lead to a solution of our functional equations for Dyck paths.

\begin{problem}
Solve the functional equations in Theorems~\ref{thm:symDyck} and~\ref{thm:svDyck}, giving generating functions for Dyck paths with respect to their degree of symmetry and their number of symmetric vertices.
\end{problem}

It is worth mentioning that the idea of allowing multiplicities (or colors) for the steps of a walk has been explored by Kauers and Yatchak in \cite{KY}, although, in their work, the multiplicity of a step does not depend on its location but only on its direction. 

Our third question asks for an extension of the results from Section~\ref{sec:unimodal} to a less artificial family of bargraphs.

\begin{problem} Extend Theorem~\ref{thm:UB} to all bargraphs, or to all unimodal bargraphs.
\end{problem}

It is natural to generalize the notion of the degree of symmetry to other classes of lattice paths. For example, one can define the degree of symmetry of a grand Motzkin path as the number of steps in the left half that are mirror images of steps in the right half, when reflecting along the vertical line through the midpoint. It is likely that the techniques from Section~\ref{sec:GD} can be extended to study this statistic on grand Motzkin paths, but the paths that play the role of the bicolored paths in Definition~\ref{def:bij} are more complicated in this case, as they consist of up and down steps of different heights, and they may jump over the $x$-axis without landing on it. Similarly, we expect that studying the degree of symmetry of Motzkin paths is more difficult than for Dyck paths, since the walks that would play the role of $\WQ^1_n$ from Section~\ref{sec:Dyck} have the added hurdle that some steps can jump over the diagonal $y=x$ without landing on it.

\begin{problem}
Find generating functions for grand Motzkin paths and Motzkin paths with respect to their degree of symmetry.
\end{problem}

The degree of symmetry can also be extended to tuples of non-crossing paths, also known as {\em watermelons}~\cite{Fisher,GOV,EliRub}. For example, in a pair of grand Dyck paths with common endpoints where one path stays weakly above the other, the degree of symmetry can be defined as the number of steps in the left half of either path that are mirror images of steps in the right half of the same path (or of either path, for a variant of the definition).

\begin{problem}
Study the distribution of the degree of symmetry on tuples of non-crossing paths.
\end{problem}

\subsection*{Acknowledgments}
The author thanks Emeric Deutsch for suggesting the notion of degree of asymmetry of combinatorial objects, Alin Bostan for helpful discussions about D-finiteness, and two reviewers for 
insightful advice on improving the content and the presentation of this paper.

\end{document}